\documentclass[10pt,onesite,draft]{article}
\newfam\cyrfam

\usepackage{amssymb,latexsym,amscd,amsmath,amsfonts, amsthm,enumerate}

\newtheorem{theorem}{Theorem}[section]

\newtheorem{proposition}[theorem]{Proposition}
\newtheorem{corollary}[theorem]{Corollary}
\newtheoremstyle{example}
  {6pt}
  {6pt}
  {}
  {}
  {\bfseries}
  {.}
  {.5em}
  {}%

\theoremstyle{example}
\newtheorem{example}[theorem]{Example}
\newtheoremstyle{definition}
  {6pt}
  {6pt}
  {}
  {}
  {\bfseries}
  {.}
  {.5em}
  {}%

\theoremstyle{definition}
\newtheorem{definition}[theorem]{Definition}

\newtheoremstyle{remark}
  {6pt}
  {6pt}
  {}
  {}
  {\bfseries}
  {.}
  {.5em}
  {}%

\theoremstyle{remark}
\newtheorem{remark}[theorem]{Remark}

\newtheoremstyle{note}
  {6pt}
  {6pt}
  {}
  {}
  {\bfseries}
  {.}
  {.5em}
  {}%

\theoremstyle{note}
\newtheorem{note}[theorem]{Note}

\makeatletter
\renewcommand\@makefntext[1]{%
\setlength\parindent{1em}%
\noindent \makebox[1.8em][r]{}{#1}} \makeatother

\begin{document}
\parskip 4pt
\large \setlength{\baselineskip}{15 truept}
\setlength{\oddsidemargin} {0.5in} \overfullrule=0mm
\def\bfh{\vhtimeb}
\date{}
\title{\bf \large The Recursion Formula for Mixed Multiplicities\\ with respect to Joint Reductions
}
\def\b{\vntime}
\author{
\begin{tabular}{ll}
 Duong Quoc Viet  \\
\small vduong99@gmail.com   \\
\end{tabular}\\
\small Department of Mathematics, Hanoi National University of Education\\
\small 136 Xuan Thuy street, Hanoi, Vietnam\\
}
 \date{}
\maketitle \centerline{
\parbox[c]{11.5cm}{
\small{\bf ABSTRACT:} This paper gives the recursion formula for mixed multiplicities of maximal degrees with respect to joint reductions of ideals, which is one of important results in the mixed multiplicity theory.
Using this result, we give consequences on the relationship between mixed multiplicities and the Hilbert-Samuel multiplicity
under most essential assumptions, that
is a substantial progress in the problem of expressing mixed
multiplicities into the Hilbert-Samuel multiplicity.
 }}

 \section{Introduction} \noindent

\footnotetext{\begin{itemize} \item[ ]{\bf Mathematics Subject
Classification (2010):} Primary 13H15. Secondary 13C15, 13D40,
14C17.  \item[ ]{\bf  Key words and phrases:} Mixed multiplicity,
 Hilbert-Samuel multiplicity,  joint reduction.
\end{itemize}}
\noindent
Let $(A, \frak{m})$ be a Noetherian local ring  with maximal ideal
$\mathfrak{m}$ and infinite residue field $k = A/\mathfrak{m}.$  Let $M$
be a finitely generated $A$-module. Let $J$ be an $\frak
m$-primary ideal, $ I_1,\ldots, I_d$ be ideals of $A.$ Put  $I
= I_1\cdots I_d;$ $\overline {M}= M/0_M: I^\infty;$ $q=\dim \overline {M}$ and
\begin{align*}
&{\bf n} =(n_1,\ldots,n_d);{\bf k} = (k_1,\ldots,k_d); {\bf
0}=(0,\ldots,0);
\\
&\mathbf{e}_i = (0, \ldots,  \stackrel{(i)}{1},  \ldots, 0)\in  \mathbb{N}^{d};|{\bf k}| = k_1 + \cdots + k_d;\\
& \mathrm{\bf I}= I_1,\ldots,I_d;  \mathrm{\bf I}^{[\mathrm{\bf
k}]}= I_1^{[k_1]},
  \ldots,I_d^{[k_d]};  \mathbb{I}^{\mathrm{\bf n}}= I_1^{n_1}\cdots I_d^{n_d}.
\end{align*}
 Assign the dimension $-\infty$ to the module $0$ and
the degree $-\infty$ to the zero polynomial. And note that $\dim \overline {M} \not = 0$ by Remark \ref{no4.3a} (iii).
Then by
\cite[Proposition 3.1]{Vi} (see \cite{MV}),
$\ell_A\Big(\dfrac{J^{n_0}\mathbb{I}^{\bf
n}M}{J^{n_0+1}\mathbb{I}^{\bf n}M}\Big)$ is a  polynomial of total
degree $q-1$ for all large enough $n_0, \bf n.$ Denote by $P(n_0, {\bf n}, J,  \mathbf{I}, M)$ this
polynomial.
And one  can
write
$$P(n_0, {\bf n}, J,  \mathbf{I}, M)= \sum_{(k_0,\; \mathbf{k})\in \mathbb{N}^{d+1}
}e(J^{[k_0+1]}, \mathbf{I}^{[\mathbf{k}]}; M)\binom{n_0 +
k_0}{k_0}\binom{\mathbf{n}+\mathbf{k}}{\mathbf{k}},$$
 where
$\binom{\mathbf{n + k}}{\bf n}= \binom{n_1 + k_1}{n_1}\cdots
\binom{n_d + k_d}{n_d}.$ Recall that the original mixed
multiplicity theory studied the mixed multiplicities $e(J^{[k_0+1]},
\mathbf{I}^{[\mathbf{k}]}; M)$ concerning the terms of highest total degree $k_0 + |\mathbf{k}| = q - 1.$
In past years, this theory has been continually developed and obtained interesting results (see e.g. [3$-$10, 13$-$35]).
Recently, in \cite{htv}, one considered a larger class than the class of original
mixed multiplicities which concerns the terms of maximal degrees in the Hilbert
polynomial, that $e(J^{[k_0+1]},
\mathbf{I}^{[\mathbf{k}]}; M)$ is  the {\it  mixed
multiplicity of
 maximal degrees of $M$ with respect   to  ideals  $J,\mathrm{\bf I}$ of the type
$(k_0+1,\mathrm{\bf k})$} if $e(J^{[h_0+1]},
\mathbf{I}^{[\mathbf{h}]}; M)=0$ for all $(h_0, \mathbf{h}) >
(k_0,\mathbf{k}).$
And a natural and pleasant picture of mixed multiplicities of maximal degrees seems to have been shown in
    \cite{htv} and \cite{TV4}. Moreover,
 \cite[Proposition 2.7]{TV4} (see \cite[Remark 4.3]{TV4}) showed that
 the mixed
multiplicity of  maximal degrees of
the type $(k_0+1,\mathrm{\bf k})$ is defined if and only if
there exists  a joint
reduction of the type
$(\mathrm{\bf k}, k_0+1).$

The concept of joint reductions gave by Rees in \cite{Re} for $\mathfrak{m}$-primary ideals,
which was extended to arbitrary ideals in \cite{Oc,{SH}, Vi2,  Vi4,  VDT,
{VT4}}.

\begin{definition}[{Definition \ref{de01}}]
Let $\frak I_i$ be a sequence  consisting $k_i$ elements of $I_i$
for all $1 \le i \le d.$
 Put  ${\bf x} = \frak I_1, \ldots, \frak
 I_d$  and $(\emptyset) = 0_A$. Then ${\bf x}$ is called a {\it joint
reduction} of $\mathbf I$ with respect to $M$ of the type ${\bf
k}=(k_1,\ldots,k_d)$ if $\mathbb{I}^{\mathbf{n} }M =
\sum_{i=1}^d(\frak I_i) \mathbb{I}^{\mathbf{n} - \mathbf{e}_i}M$
for all large $\bf n.$  If $d=1$ then $(\frak I_1)$
is called a {\it reduction} of $I_1$ with respect to $M$
\cite{NR}.
\end{definition}

The first purpose of this paper is to give the recursion formula for mixed multiplicities of maximal degrees
in a form convenient for applications. This is one of motivations to lead us to the following recursion formula for mixed multiplicities of maximal degrees of modules with respect to joint reductions of ideals.

\begin{theorem}[{Theorem \ref{le2020}}]\label{thm1.1}
 Let ${\bf x}= x_1, \ldots, x_n$ be  a joint reduction  of
$\mathbf{I}, J$ with respect to $M$
 of the type $({\bf
k},k_0+1).$ Assume that $k_i > 0$ and $x_1 \in I_i$ for $1 \le i \le d.$ Then  $e(J^{[k_0+1]}, \mathbf{I}^{[\mathbf{k}]}; M) =
e(J^{[k_0+1]}, \mathbf{I}^{[\mathrm{\bf k} - \mathbf{e}_i]};
 M/x_1M) - e(J^{[k_0+1]}, \mathbf{I}^{[\mathrm{\bf k} - \mathbf{e}_i]};
 0_M: x_1).$
\end{theorem}

Perhaps Theorem \ref{thm1.1} is one of expected results in studying mixed multiplicities.
And one of the interesting consequences of this theorem is the following.
\begin{corollary} [{Corollary \ref{vt5/3}}]\label{co 1.4}
 Let ${\bf x}= x_1, \ldots, x_n$ be  a joint reduction  of
$\mathbf{I}, J$ with respect to $M$
 of the type $({\bf
k},k_0+1).$  Assume that $k_i > 0$ and $x_1 \in I_i$ for $1 \le i \le d.$ Then we have
\begin{itemize}
\item[$\mathrm{(i)}$]
$e(J^{[k_0+1]}, \mathbf{I}^{[\mathbf{k}]}; M)\le
e(J^{[k_0+1]}, \mathbf{I}^{[\mathrm{\bf k} - \mathbf{e}_i]};
 M/x_1M).$
  \item[$\mathrm{(ii)}$] $e(J^{[k_0+1]}, \mathbf{I}^{[\mathbf{k}]}; M)=
e(J^{[k_0+1]}, \mathbf{I}^{[\mathrm{\bf k} - \mathbf{e}_i]};
 M/x_1M)$  if $x_1$ is an $M$-regular element.
 \item[$\mathrm{(iii)}$] $e(J^{[k_0+1]}, \mathbf{I}^{[\mathbf{k}]}; M)=
e(J^{[k_0+1]}, \mathbf{I}^{[\mathrm{\bf k} - \mathbf{e}_i]};
 M/x_1M)$ if $x_1$ is an $I$-filter-regular element with respect to $M.$
\end{itemize}
\end{corollary}

Note that the equation in Corollary \ref {co 1.4} (iii) plays a very important role in studying mixed multiplicities.
And to have this equation, one of the approaches
is to choose $x \in I_i$ such that $P(n_0, {\bf n}, J, \mathbf{I}, M/xM)=\bigtriangleup^{(0,\; \mathbf{e}_i)}P(n_0, {\bf n}, J,
\mathbf{I},  M).$ So one used different sequences. However, these sequences are both filter-regular sequences and Rees superficial sequences
and are also parts of joint reductions (see Remark \ref{no4.2a}).

Corollary \ref {co 1.4} immediately produces the following result on the relationship between mixed multiplicities and the Hilbert-Samuel multiplicity via joint reductions.
 \begin{corollary} [{Corollary \ref{vt15/7}}]\label{co 1.vv}  Let ${\bf x}= x_1, \ldots, x_n$ be  a joint reduction  of
$\mathbf{I}, J$ with respect to $M$
 of the type $({\bf
k},k_0+1)$ with
${\bf x}_{\mathbf{I}} =x_1,\ldots,x_{|{\bf k }|} \subset {\bf I};$ $U =x_{|{\bf k }|+1},\ldots,x_n \subset J.$
 Then we have
\begin{itemize}
\item[$\mathrm{(i)}$] $e(J^{[k_0 +1]}, \mathbf{I}^{[\mathbf{k}]}; M)
\le e(U; {M}/({\bf x}_{\mathbf{I}})M:I^\infty).$
\item[$\mathrm{(ii)}$] $e(J^{[k_0 +1]}, \mathbf{I}^{[\mathbf{k}]}; M)
= e(U; {M}/({\bf x}_{\mathbf{I}})M:I^\infty)$ if ${\bf x}_{\mathbf{I}}$ is an $I$-filter-regular sequence with respect to $M.$
\end{itemize}
\end{corollary}
 Thus we obtained two formulas under essential assumptions.
   Moreover,
   Corollary \ref {co 1.vv} (ii), not
 only includes all cases
but also shows that $e(J^{[k_0 +1]}, \mathbf{I}^{[\mathbf{k}]}; M) > 0$ if and only if $\dim {M}/({\bf x}_{\mathbf{I}})M:I^\infty = k_0 +1.$
 Hence one could say that it seems to be a perfect transition of mixed multiplicities into the Hilbert-Samuel multiplicity.

Moreover, using Corollary \ref {co 1.4}, we prove the following result.

\begin{corollary}[{Corollary \ref{vt2019}}]\label{thm1.3}  Let ${\bf x}= x_1, \ldots, x_n$ be  a joint reduction  of
$\mathbf{I}, J$ with respect to $M$
 of the type $({\bf
k},k_0+1)$ with
${\bf x}_{\mathbf{I}} =x_1,\ldots,x_{|{\bf k }|} \subset {\bf I}.$ Assume that ${\bf x}$ is a system of parameters for ${M}.$
Set $I = I_1\cdots I_d.$ Then
\begin{itemize}
\item[$\mathrm{(i)}$] $e(J^{[k_0 +1]}, \mathbf{I}^{[\mathbf{k}]}; M)
\le e(\mathbf{x}; {M}).$
\item[$\mathrm{(ii)}$] $e(J^{[k_0 +1]}, \mathbf{I}^{[\mathbf{k}]}; M)
= e(\mathbf{x}; {M})$ if
 $\dim {M}/({\bf x}_{\mathbf{I}},I){M} < \dim
{M}/({\bf x}_{\mathbf{I}}){M}.$
\end{itemize}
\end{corollary}

Recall that Risler-Teissier \cite{Te} in 1973 proved that any mixed multiplicity is
the Hilbert-Samuel multiplicity of a superficial sequence, and  Rees in 1984 \cite{Re}
showed that each mixed multiplicity is
the multiplicity of a joint reduction. However, their
results only deal with the case where all ideals are $\mathfrak{m}$-primary.
Obviously, in general, to give results which are analogous to Risler-Teissier's and Rees's theorems, it is not always possible. So one gave sufficient conditions under which mixed
multiplicities are the Hilbert-Samuel multiplicity of joint
reductions (see \cite{VDT, TV3}).
In this context, Corollary \ref{thm1.3} (ii) is the strongest result, because on the one hand it covers the results  in
 \cite{Re, Te, TV3, VDT} (see Remark \ref{no4.3b}), on the other hand it also yields the following.

 \begin{corollary}[Corollary \ref{cr2019}]\label{cr3}  Let ${\bf x}= x_1, \ldots, x_n$ be  a joint reduction  of
$\mathbf{I}, J$ with respect to $M$
 of the type $({\bf
k},k_0+1)$ with ${\bf x}_{\mathbf{I}} =x_1,\ldots,x_{|{\bf k }|} \subset {\bf I}.$ Assume that $\mathrm{ht}\frac{\mathrm{Ann}[M/({\bf x}_{\mathbf{I}})M]+I}{\mathrm{Ann}[M/({\bf x}_{\mathbf{I}})M]} > 0$ and $k_0 + |{\bf k}| = \dim \overline{M}-1.$
Then
${\bf x}$ is a system of parameters for ${M}$ and $e(J^{[k_0 +1]}, \mathbf{I}^{[\mathbf{k}]}; M)
= e(\mathbf{x}; {M}).$
\end{corollary}

    Corollary \ref{thm1.3} and Corollary \ref{cr3} seem to make the
problem of expressing mixed multiplicities into the Hilbert-Samuel
multiplicity of joint reductions become clear.

Returning now to Theorem \ref{thm1.1}, applying this theorem we easily recover early results on mixed multiplicities of ideals (see Remark \ref{no4.2a}). Moreover, the assumption of the theorem seems to be minimal.
 So we hope that
it would become an effective tool.
In fact, Theorem \ref{thm1.1} yields Corollary \ref {co 1.4}; Corollary \ref{co 1.vv} and Corollary \ref{thm1.3} under most essential assumptions.
 Different from the previous approaches, our approach is based on the Euler-Poincar\'{e} characteristic of joint
reductions. We give the recursion formula for this characteristic in Proposition \ref{pru2020} that is a key tool in
this paper.

The paper is divided into four sections. Section 2 deals with polynomials
 called temporarily the Euler-Poincar\'{e} polynomial of graded modules, and gives the recursion formula for these polynomials (Proposition \ref{lem 3.4}).
Section 3 is devoted to the discussion of the
Euler-Poincar\'{e} characteristic of joint
reductions of ideals, and proves the recursion formula for this characteristic with respect to joint reductions (Proposition \ref{pru2020}).
In Section 4, we prove Theorem \ref{thm1.1} and corollaries for mixed multiplicities of ideals. The paper is ended by Example \ref{exam} and Remark \ref{no4.2a}.

\section{ The Euler-Poincar\'{e} polynomial of joint
reductions}

This section first introduces  polynomials
 called temporarily
the Euler-Poincar\'{e} polynomial of joint
reductions, and gives the recursion formula for these polynomials.

Let $(A, \frak{m})$ be a Noetherian local ring  with maximal ideal
$\mathfrak{m}$ and infinite residue field $k = A/\mathfrak{m}.$  Let $J$ be an $\frak
m$-primary ideal. Set
\begin{align*}
&{\bf
0}=(0,\ldots,0); \mathbf{e}_i = (0, \ldots,  \stackrel{(i)}{1},  \ldots, 0)\in  \mathbb{N}^{d}; \\
&{\bf n} =(n_1,\ldots,n_d); {\bf k} =(k_1,\ldots,k_d) \in  \mathbb{N}^{d};\\
&{\bf 1}=(1,\ldots,1) \in  \mathbb{N}^{d}; |{\bf k}| = k_1 + \cdots + k_d.
\end{align*}
  Let
$S= \bigoplus_{\mathbf{n} \geq \mathbf{0}}S_{\mathbf{n}}$
 be a finitely generated standard $\mathbb{N}^d$-graded algebra over $A$ (i.e., $S$ is generated over $A$
  by elements of total degree 1); $V=\bigoplus_{\bf n\ge \bf 0}V_{\bf n}$
  be  a finitely generated $\mathbb{N}^d$-graded $S$-module.
  Set  $S_i= S_{{\bf e}_i}$
  for all $ 1 \le i \le d,$ $\mathbf{S} = S_1, \ldots, S_d.$

  Now, we would like to recall the following
concept of joint reductions (see \cite{KR1, TV4}).

\begin{definition}\label{de 6/7} Set $J=S_0.$ Let $\frak R_i$ be a sequence  consisting $k_i$
elements of $S_i$ for all $0 \le i \le d$ and $k_0,\ldots,k_d \ge
0.$
 Put  ${\frak R} = \frak R_1, \ldots, \frak
 R_d, \frak R_0$  and $(\emptyset) = 0_S$.
 Then $\frak R$ is called a {\it joint
reduction} of $\mathbf{S}, J$ with respect to $V$ of the type
$({\bf k}, k_0)$ if
$J^{n_0}V_\mathbf{n} = (\frak
R_0)J^{n_0-1}V_{\mathbf{n}} +\sum_{i=1}^d(\frak R_i) J^{n_0
}V_{\mathbf{n} - \mathbf{e}_i} \mathrm{\; for \; all \; large}\; n_0, \bf n.$
\end{definition}

\begin{note}\label{n2} Let ${\frak R} = \frak R_1, \ldots, \frak
 R_d, \frak R_0$  be a joint reduction of $\mathbf{S},
J$ with respect to $V$ with $x_1 \in \frak R_1.$ Set $\frak R'_1 = \frak R_1 \setminus \{x_1\}$ and ${\frak R'} = \frak R'_1,\frak R_2, \ldots, \frak
 R_d, \frak R_0$. Set $Q =JS_{\bf 1}$ and $H =(\frak R_0)S_{\bf 1} + \sum_{i=1}^dJ(\frak R_i)S_{{\bf 1-e}_i};$ $V({\bf r})=\bigoplus_{\bf n\ge \bf r}V_{\bf n}.$
 Then we obtain the following.
\begin{enumerate}[(i)]
\item      Note that $S_{\bf 1}= S_1\cdots S_d$ and $S = A[S_1,\ldots,S_d].$
Since $S$ is a finitely generated standard graded algebra over $A$ and
 $V$ is a finitely generated graded $S$-module,
   by \cite[Lemma 2.3]{Vi6}, there exists $\bf r \ge \bf 0$  such that
$V_{\bf n + \bf r}= S_{\bf n}V_{\bf r}$ for all $\bf n \ge 0.$ In this case,
$V({\bf r})= SV_{\bf r}.$
  This implies that
${\frak R}$ is a joint reduction of $\mathbf{S},
J$ with respect to $V$  if and only if $(H)$ is a reduction of $(Q)$ with respect to $V({\bf r})$ for some $\bf r,$ i.e., $(Q)^{m+1}V({\bf r}) =  (H)(Q)^{m}V({\bf r})$ for all large $m$ \cite{NR}.

\item
 Let $x_1 \in \sqrt{\mathrm{Ann}_SV}.$
  By (i), $(H)$ is a reduction of $(Q)$ with respect to $V({\bf r})$ for some $\bf r.$ Hence, by \cite[Lemma 8.1.4]{SH},
$(H)$ is a reduction of $(Q)$ with respect to $S/P$ for all $P \in \mathrm{Min}_S{V({\bf r})}.$
Now since $x_1 \in \sqrt{\mathrm{Ann}_SV}$ and $\mathrm{Ann}_SV \subset \mathrm{Ann}_SV({\bf r}),$ it follows that $x_1 \in P$ for all $P \in \mathrm{Min}_S{V({\bf r})}.$ Hence
$$((\frak R_0)S_{\bf 1} + J(\frak R'_1)S_{{\bf 1-e}_1}+ \sum_{i=2}^dJ(\frak R_i)S_{{\bf 1-e}_i})$$ is a reduction of  $(Q)$
with respect to $S/P$ for all $P \in \mathrm{Min}_S{V({\bf r})}.$ Therefore, $((\frak R_0)S_{\bf 1} + J(\frak R'_1)S_{{\bf 1-e}_1}+ \sum_{i=2}^dJ(\frak R_i)S_{{\bf 1-e}_i})$ is a reduction of  $(Q)$ with respect to $V({\bf r})$ by \cite[Lemma 8.1.4]{SH}.
So by (i), ${\frak R'}$ is also a joint reduction of $\mathbf{S},
J$ with respect to $V.$

\item
Let $V'=\bigoplus_{\bf n\ge \bf 0}V'_{\bf n}$
  be  a graded $S$-submodule of $V$.
  Note that  by (i), $(H)$ is  a reduction of  $(Q)$
with respect to
$V({\bf r})$ for some $\bf r.$ Set $V'({\bf r})= \bigoplus_{\bf n\ge \bf r}V'_{\bf n}.$
Since $V'({\bf r})\subseteq V({\bf r}),$
    it follows that
$(H)$ is  a reduction of  $(Q)$
with respect to $V'({\bf r})$ because
$\mathrm{Ann}_SV'({\bf r}) \supset \mathrm{Ann}_SV({\bf r}))$ (see e.g \cite[Lemma 17.1.4]{SH}).
From this it follows that  ${\frak R}$ is also a joint reduction of $\mathbf{S},
J$ with respect to $V'$ by (i).

\item Note that $(0_V : x_1)$ is a graded $S$-submodule of $V.$ Hence by (iii), ${\frak R}$ is a joint reduction of $\mathbf{S},
J$ with respect to $(0_V : x_1).$ Therefore since $x_1(0_V : x_1) = 0_V,$ it follows that ${\frak R'}$ is a joint reduction of $\mathbf{S},
J$ with respect to $(0_V : x_1).$
\end{enumerate}
\end{note}

Let $t$ be a variable over $A.$ Then with the notations $S$ and $V$ as the above, we have
the $\mathbb{N}^{d+1}$-graded algebra $\mathcal{S} = S[t] =
\bigoplus_{n_0\geqslant 0,\bf n\geqslant 0}S_{\mathrm{\bf
n}}t^{n_0}$ and the $\mathbb{N}^{d+1}$-graded
$\mathcal{S}$-modules  $\mathcal{V} = \bigoplus_{n_0\geqslant
0,\bf n \geqslant 0}V_{\mathrm{\bf n}}t^{n_0}\;
\text{ and }\;\mathcal{V}' = \bigoplus_{n_0\geqslant 0,\bf n
\geqslant 0}J^{n_0}V_{\mathrm{\bf n}}t^{n_0}.$

Let ${\frak R} = \frak R_1, \ldots, \frak
 R_d, \frak R_0$  be a joint reduction of $\mathbf{S},
J$ with respect to $V$ of the type $({\bf k}, k_0)$, where $
\mathfrak{R}_0 \subset J, \mathfrak{R}_i \subset S_i$ for all $1 \le i \le d.$ Put $\frak R_0t = \{at \mid a \in \frak R_0 \};$
$X= \frak R_1, \ldots, \frak
 R_d, \frak R_0t;$  $n  = k_0 + |{\bf k}|.$

 In \cite{TV4}, one considered the Koszul complex of
$\mathcal{V}$ with respect to $X:$
$$ 0\longrightarrow K_n(X,\mathcal{V})\longrightarrow K_{n-1}(X,\mathcal{V})\longrightarrow
\cdots\longrightarrow K_1(X,\mathcal{V})\longrightarrow
K_0(X,\mathcal{V}) \longrightarrow 0$$ and obtained the homology modules
$H_0({X},\mathcal{V}),
 H_1({X},\mathcal{V}),\ldots,H_n({X},\mathcal{V}).$
  By applying \cite[Theorem 4.2]{KR1}, \cite{TV4} showed that $\ell_A\big(H_i({
X},\mathcal{V})_{{(n_0,{\bf n})}}\big) < \infty$ for all large enough
$n_0, \bf n$ and for  all $0\le i \le n,$ and then
  $\sum_{i=0}^n (-1)^i \ell_A\big(H_i({ X},\mathcal{V})_{{(n_0,{\bf
n})}}\big)$ is a polynomial in $n_0, \bf n$ for all large  enough $n_0,
\bf n.$ This polynomial is called temporarily the {\it Euler-Poincar\'{e} polynomial of $\frak R,$} and
denoted by $
\chi(n_0,\mathbf{n},\frak R, J, V).$

Then by an argument analogous to what used for the proof of \cite[Lemma 3.3] {VT3} (see the proof of \cite[Theorem 4.7.4] {BH1}), in this context,
 we get the following proposition.
\vskip 0.2cm
\begin{proposition}\label{lem 3.4} Let $\frak R = x_1,\ldots,x_n$ be a joint
reduction of $\mathbf{S}, J$ with respect to $V.$ Set $\frak R' = x_2,\ldots,x_n.$
Then the following statements hold.
\begin{itemize}
\item[$\mathrm{(i)}$] If $x_1 \in
\sqrt{\mathrm{Ann}_SV},$ then $
\chi(n_0,\mathbf{n},\frak R, J, V)=0.$
\item[$\mathrm{(ii)}$] If $x_1$ is a $V$-regular element, then $
\chi(n_0,\mathbf{n},\frak R, J, V)=
\chi(n_0,\mathbf{n},\frak R', J, V/x_1V).$
\item[$\mathrm{(iii)}$] $\chi(n_0,\mathbf{n},\frak R, J, V) = \chi(n_0,\mathbf{n},\frak R', J, V/x_1V)-\chi(n_0,\mathbf{n},\frak R', J, 0_V:x_1).$
\end{itemize}
\end{proposition}

\begin{proof}
Assume that ${\frak R} = \frak R_1, \ldots, \frak
 R_d, \frak R_0,$ where $ \mathfrak{R}_0 \subset J, \mathfrak{R}_i \subset S_i$ for all $1 \le i \le d.$ Put $\frak R_0t = \{at \mid a \in \frak R_0 \};$
$X= \frak R_1, \ldots, \frak
 R_d, \frak R_0t.$ And without loss of generality, we may assume that $x_1 \in X.$ Set $X' = X \backslash \{x_1\}.$ Since $\frak R$ is a joint
reduction of $\mathbf{S}, J$ with respect to $V,$ it follows that $\frak R'$ is a joint
reduction of $\mathbf{S}, J$ with respect to $V/x_1V.$

The proof of (i): First we consider the case that $x_1 \in \mathrm{Ann}_SV.$  By \cite[Corollary 1.6.13 (a)]{BH1}, we have the following exact sequence
\begin{equation}\label{v02020}\cdots\xrightarrow{\pm x_1} H_i(X',\mathcal{V}) \longrightarrow H_{i}(X,\mathcal{V})\longrightarrow H_{i-1}(X',\mathcal{V})\xrightarrow{\pm x_1} H_{i-1}(X',\mathcal{V})\longrightarrow \cdots.
\end{equation}
Observe that since $x_1\mathcal{V} = 0,$ it follows that
 $x_1$ annihilates $H_{i}(X',\mathcal{V})$ for all $i.$  Consequently, the exact sequence (\ref{v02020}) breaks up into the following exact sequences
$$0\longrightarrow H_i(X',\mathcal{V}) \longrightarrow H_{i}(X,\mathcal{V})\longrightarrow H_{i-1}(X',\mathcal{V})\longrightarrow 0$$
for all $i.$ Moreover, since $x_1\mathcal{V} = 0,$ it follows that
$X'$ is also a joint
reduction of  $\mathcal{V}.$  So
  $$\ell_A[H_i({X},\mathcal{V})_{(n_0,\mathbf{n})}] = \ell_A[H_i(X',\mathcal{V})_{(n_0,\mathbf{n})}] + \ell_A[H_{i-1}(X',\mathcal{V})_{(n_0,\mathbf{n})}]$$
for all $i$ and for all large  $n_0,\mathbf{n}$. Note  that
$H_{n + 1}(X, \mathcal{V}) = H_{-1}(X,\mathcal{V}) = 0.$ Hence
$$\begin{aligned}&\sum_{i=0}^{n+1} (-1)^i \ell_A(H_i({X},\mathcal{V})_{(n_0,\mathbf{n})})\\
&= \sum_{i=0}^{n+1} (-1)^i [\ell_A(H_i(X',\mathcal{V})_{(n_0,\mathbf{n})}) + \ell_A(H_{i-1}(X',\mathcal{V})_{(n_0,\mathbf{n})})] = 0 \end{aligned}$$
for all large $n_0,\mathbf{n}.$ So $\chi(n_0,\mathbf{n},\frak R, J, V)= 0.$ This implies also that for any $x \in \frak R,$
 $\chi(n_0,\mathbf{n},\frak R, J, V/xV)= 0.$
We turn now to the case that $x_1 \in \sqrt{\mathrm{Ann}_SV}.$ By Note \ref{n2} (ii), $\frak R'$ is also a joint
reduction of $\mathbf{S}, J$ with respect to $V.$
And there exists $u > 0$ such that
$x_1^u V = 0_V.$
Recall that by \cite[Remark 3.4 (i)]{TV4},
 \begin{equation}\label{v021}  \chi(n_0,\mathbf{n},\frak R, J, \underline{\;\;}\,) \text{ is additive on  exact sequences of S-modules.}
\end{equation}
Since $\chi(n_0,\mathbf{n},\frak R, J, V/x_1V)= 0$ and the exact sequence
$$0\longrightarrow x_1{V} \longrightarrow {V} \longrightarrow {V}/x_1{V}\longrightarrow 0,$$ it follows by (\ref{v021}), that
$\chi(n_0,\mathbf{n},\frak R, J, V)= \chi(n_0,\mathbf{n},\frak R, J, x_1V).$
Therefore, $$\chi(n_0,\mathbf{n},\frak R, J, V)= \chi(n_0,\mathbf{n},\frak R, J, x_1^jV)$$ for all $j \ge 1.$ Consequently, $\chi(n_0,\mathbf{n},\frak R, J, V)=\chi(n_0,\mathbf{n},\frak R, J, x_1^uV) = 0.$
 We get (i).

The proof of (ii): Since  $x_1$ is $V$-regular, it follows that $x_1$ is $\mathcal{V}$-regular. Hence
 $$H_i(X, \mathcal{V}) \cong H_i(X', \mathcal{V}/x_1\mathcal{V})$$ by \cite[Corollary 1.6.13 (b)]{BH1}. Thus
$\ell_A(H_i({X},\mathcal{V})_{(n_0,\mathbf{n})}) = \ell_A(H_i({X}',\mathcal{V}/x_1\mathcal{V})_{(n_0,\mathbf{n})})$
for all $n_0,\mathbf{n}$ and for all $i$. Therefore we have (ii).

  The proof of (iii): Consider the following cases.

If  $x_1 \in \sqrt{\mathrm{Ann}_SV},$ then $\frak R'$ is a joint
reduction of $\mathbf{S}, J$ with respect to $V$ by Note \ref{n2} (ii).
 In this case, $\chi(n_0,\mathbf{n},\frak R, J, V)=0$ by (i).
On the other hand, ${\frak R'}$ is also a joint reduction of $\mathbf{S},
J$ with respect to $0_V : x_1$ by Note \ref{n2} (iv). Hence
since the exact sequence
$$0\longrightarrow (0_{{V}}:x_1) \longrightarrow {V} \xrightarrow{\;\;x_1\;\;} {V}\longrightarrow {V}/x_1{V}\longrightarrow 0,$$
we have $\chi(n_0,\mathbf{n},\frak R', J, V/x_1V) - \chi(n_0,\mathbf{n},\frak R', J, 0_V: x_1) = 0$ by (\ref{v021}).
So $$\chi(n_0,\mathbf{n},\frak R, J, V) =\chi(n_0,\mathbf{n},\frak R', J, V/x_1V) - \chi(n_0,\mathbf{n},\frak R', J, 0_V: x_1).$$

Set set $ E = 0_{V}: x_1^{\infty}.$ By Note \ref{n2} (iii), ${\frak R}$ is also a joint reduction of $E.$
Let $x_1 \notin \sqrt{\mathrm{Ann}_SV}.$ Then $x_1$ is $(V/E)$-regular. So $ x_1V \cap E = x_1E.$
Therefore we get
 the exact sequence
 $$0\longrightarrow {E}/{x_1E}\longrightarrow {V}/{x_1 V} \longrightarrow {V}/({x_1 V +E}) \longrightarrow 0.$$
  Hence by (\ref{v021}),
 \begin{equation}\label{v00} \chi(n_0,\mathbf{n}, {\frak R}', {V}/({x_1V+E})) = \chi(n_0,\mathbf{n}, {\frak R}',
{V}/{x_1V})- \chi(n_0,\mathbf{n}, {\frak R}', {E}/{x_1E}).
\end{equation}
Since $V$ is Noetherian, it follows that $E = 0_{ V
}: x_1^v$ for
some $v > 0.$ Consequently $x_1^vE = 0_V.$ Hence $x_1  \in \sqrt{\mathrm{Ann}_{S}E}.$ So by Note \ref{n2} (ii), ${\frak R}'$ is a joint reduction of $E.$ Recall that  ${\frak R}'$ is also a joint reduction of $0_{ V
}:x_1$ by Note \ref{n2} (iv).
Note that $(0_{ V
}:x_1) \subset E.$ Then we have
 the exact sequence
$$0\longrightarrow (0_{ V
}:x_1) \longrightarrow E \xrightarrow{\;\;x_1\;\;} E\longrightarrow {E}/{x_1E}\longrightarrow 0.$$
Therefore by (\ref{v021}), $\chi(n_0,\mathbf{n}, {\frak R}', {E}/{x_1E}) = \chi(n_0,\mathbf{n}, {\frak R}', 0_V:x_1).$
So  by (\ref{v00}), we get
\begin{equation}\label{v020}\chi(n_0,\mathbf{n}, {\frak R}', {V}/({x_1V+E})) = \chi(n_0,\mathbf{n}, {\frak R}',
{V}/{x_1V})- \chi(n_0,\mathbf{n}, {\frak R}', 0_V: x_1).
\end{equation}
Since $x_1$ is  $(V/E)$-regular, by (ii),
$\chi(n_0,\mathbf{n}, {\frak R},
{V}/{E}) = \chi(n_0,\mathbf{n}, {\frak R}', {V}/({x_1V+E})).$
 On the other hand, we always have $$\chi(n_0,\mathbf{n}, {\frak R}, {V}/{E})=
\chi(n_0,\mathbf{n}, {\frak R}, V)-\chi(n_0,\mathbf{n}, {\frak R}, E)$$ by (\ref{v021}). Consequently  we
obtain
$$\chi(n_0,\mathbf{n}, {\frak R}', {V}/({x_1V+E})) = \chi(n_0,\mathbf{n}, {\frak R}, V)-\chi(n_0,\mathbf{n}, {\frak R}, E).$$
Now, since $x_1  \in \sqrt{\mathrm{Ann}_SE},$ it follows that $\chi(n_0,\mathbf{n}, {\frak R}, E) = 0$ by (i).  So
\begin{equation}\label{vv020}\chi(n_0,\mathbf{n}, {\frak R}', {V}/({x_1V+E})) = \chi(n_0,\mathbf{n}, {\frak R}, V).
\end{equation}
Finally, by (\ref{v020}) and (\ref{vv020}) we get  $$\chi(n_0,\mathbf{n}, {\frak R}, V) =\chi(n_0,\mathbf{n}, {\frak R}', V/x_1V)-
\chi(n_0,\mathbf{n}, {\frak R}', 0_V: x_1).$$
The proof is complete.\end{proof}

 Now, denote by $F(n_0, {\bf n}, J, V)$ the Hilbert polynomial of the Hilbert function $\ell_A\big((\mathcal{V}/\mathcal{V}')_{(n_0,{\bf n})}\big) =
\ell_A({V_{\mathbf{n}}}/{J^{n_0}V_{\mathbf{n}}}\big),$ and
   $\bigtriangleup^{(k_0,\;\mathbf{k})}F(n_0, {\bf n}, J, V)$ the $(k_0,\bf
k)$-difference of the polynomial $F(n_0, {\bf n}, J, V)$.
Then by \cite[Lemma 3.3]{TV4},
 $ \chi(n_0,\mathbf{n},\frak R, J, V)$ is a constant if and only if  $\bigtriangleup^{(k_0,\;\mathrm{\bf k})}F(n_0, {\bf n},J,  V)$ is a constant and in this case,
 \begin{equation}\label{v2020}
 \chi(n_0,\mathbf{n},\frak R, J, V) =\bigtriangleup^{(k_0,\;\mathrm{\bf k})}F(n_0, {\bf n}, J,  V).\end{equation}

\section{The Euler-Poincar\'{e} characteristic of joint
reductions}

This section is devoted to the discussion of the
Euler-Poincar\'{e} characteristic of joint
reductions of ideals, and gives the recursion formula for this characteristic with respect to joint reductions.

Let $(A, \frak{m})$ be a Noetherian local ring  with maximal ideal
$\mathfrak{m}$ and infinite residue field $k = A/\mathfrak{m}.$ Let $M$
be a finitely generated $A$-module. Let $ \mathrm{\bf I}= I_1,\ldots, I_d$ be ideals of $A.$
 Recall that  $I
= I_1\cdots I_d,$  and
 $\mathbb{I}^{\mathrm{\bf n}}= I_1^{n_1}\cdots I_d^{n_d}.$
   Let $t_1, \ldots, t_d$ be  variables over $A$. Set $T=t_1,
\ldots, t_d$ and $T^{\bf n}=t_1^{n_1}\cdots t_d^{n_d};$ $ \mathbf{e}_i = (0, \ldots,  \stackrel{(i)}{1},  \ldots, 0)\in  \mathbb{N}^{d}$
for all $1 \le i \le d;$ ${\bf 1}=(1,\ldots,1) \in  \mathbb{N}^{d}.$ Let $J$ be an $\frak
m$-primary ideal.
Denote by
$$\frak R(\mathrm{\bf I}; A) =
  \bigoplus_{\bf n\geqslant 0}\mathbb{ I}^{\mathrm{\bf n}}T^{\bf n}\; \text{ and }\; \frak R(\mathrm{\bf I}; M) =
   \bigoplus_{\bf n \geqslant 0}\mathbb{ I}^{\mathrm{\bf n}}MT^{\bf n} $$
  the Rees algebra of $\mathbf{I}$ and  the  Rees module of $\mathbf{I}$ with respect to $M$,
  respectively.

Recall that the notion of joint reductions in Rees's work in 1984 \cite{Re} first was developed
by O'Carroll in 1987 \cite{Oc}, and which was studied
in \cite{SH, Vi2,  Vi4,  VDT, VT4}.

\begin{definition} \label{de01}
Let $\frak I_i$ be a sequence  consisting $k_i$ elements of $I_i$
for all $1 \le i \le d$ and $k_1,\ldots,k_d \ge 0.$
 Put  ${\bf x} = \frak I_1, \ldots, \frak
 I_d$  and $(\emptyset) = 0_A$. Then ${\bf x}$ is called a {\it joint
reduction} of $\bf I$ with respect to $M$ of the type ${\bf k}
=(k_1,\ldots,k_d)$ if $\mathbb{I}^{\mathbf{n} }M =
\sum_{i=1}^d(\frak I_i) \mathbb{I}^{\mathbf{n} - \mathbf{e}_i}M$
for all large $\bf n.$  If $d=1$ then $(\frak I_1)$
is called a {\it reduction} of $I_1$ with respect to $M$
\cite{NR}.
\end{definition}

Let ${\bf x} = \mathfrak{I}_1, \ldots,
\mathfrak{I}_d, \mathfrak{I}_0$  be a joint reduction of
$\mathbf{I}, J$ with respect to $M$ of the type $({\bf k},
k_0+1)$, where $ \mathfrak{I}_0 \subset J,  \mathfrak{I}_i \subset
I_i$ for all $1 \le i \le d.$ Set $\mathbf{I}_T = I_1t_1, \ldots, I_dt_d$
and $${\bf x}_T = \mathfrak{I}_1t_1, \ldots,
\mathfrak{I}_dt_d, \mathfrak{I}_0.$$ Then it is easily seen
that $\mathbf{x}_T$ is  a joint reduction of $\mathbf{I}_T, J$
with respect to $\frak R(\mathbf{I}; M)  $ of the type $({\bf k},
k_0+1)$.
Denote by $F(n_0, \mathbf{n}, J, M)$ and $P(n_0, {\bf n}, J, \mathbf{I}, M)$ the Hilbert polynomials of the Hilbert functions
$\ell_A\Big(\frac{
  \mathbb{I}^{\mathbf{n}}M}{J^{n_0}\mathbb{I}^\mathbf{n}M}\Big)$ and $\ell_A\Big(\dfrac{J^{n_0}\mathbb{I}^{\bf
n}M}{J^{n_0+1}\mathbb{I}^{\bf n}M}\Big),$ respectively. Then it is obviously that
    $\bigtriangleup^{(k_0+1,\;\mathrm{\bf
k})}F(n_0, {\bf n}, J,  M) =
\bigtriangleup^{(k_0, \;\mathrm{\bf k})}P(n_0, {\bf n}, J,
\mathbf{I}, M)$ for any $(k_0, \mathbf{k})\in \mathbb{N}^{d+1}.$
Since $\bf x$ is a joint reduction of $\mathbf{I}, J$
of the type $({\bf k}, k_0+1)$, we get  by \cite[Proposition 2.2]{TV4}, that
$\bigtriangleup^{(k_0,\;\mathrm{\bf k})}P(n_0, {\bf n}, J,
\mathbf{I}, M)$ is a constant. Hence
$\bigtriangleup^{(k_0+1,\;\mathrm{\bf k})}F(n_0, {\bf n}, J, \mathbf{I}, M)$ is a constant. Therefore by (\ref{v2020}), $ \chi(n_0,\mathbf{n},\mathbf{x}_T, J, \frak
R(\mathbf{I}; M))$ is also a constant. In this case, put
\begin{equation}\label{vv8/3}\chi(n_0,\mathbf{n},\mathbf{x}_T, J, \frak
R(\mathbf{I}; M))=\chi(\mathbf{x}, J, \mathbf{I}, M),\end{equation} and $\chi(\mathbf{x}, J, \mathbf{I}, M)$ is called the {\it
Euler-Poincar\'{e} characteristic of $\bf x$
  with respect to $J,\bf I$ and $M$ of the type $(k_0+1,\mathrm{\bf
  k})$} (see \cite[Note 3.5]{TV4}). And \cite{TV4} yields also the following.

\begin{remark} \label{ree2020} We have the following facts.
\begin{itemize}
\item[$\mathrm{(i)}$] By \cite[Remark 3.4 (iii)]{TV4}, $\chi(\mathbf{x}, J, \mathbf{I}, M) \ge 0.$ And by \cite[Theorem 4.4 (i)]{TV4}, $\chi(\mathbf{x}, J, \mathbf{I}, M)=\bigtriangleup^{(k_0,\;\mathrm{\bf k})}P(n_0, {\bf n}, J,
\mathbf{I}, M).$
 \item[$\mathrm{(ii)}$]
By \cite[Proposition 3.7 (iii)]{TV4}, $\chi(\mathbf{x}, J, \mathbf{I}, M) \ge \chi(\mathbf{x}, J, \mathbf{I}, M')$ for any  submodule $M'$ of $M.$
\item[$\mathrm{(iii)}$] By \cite[Remark 3.6 (ii)]{TV4}, $\chi(\mathbf{x}, J, \mathbf{I}, I^kM) = \chi(\mathbf{x}, J, \mathbf{I}, M)$ for any  integer $k>0.$
\end{itemize}
\end{remark}

Let $a \in \frak m.$ Recall that $a$ is called an $M$-{\it regular element} if $0_M: a = 0_M.$ And
$a$ is called an {\it $I$-filter-regular element with respect to $M$} if
$0_M:a \subseteq 0_M: I^{\infty}.$

The following proposition is the main tool
for our approach in this paper.

\begin{proposition} \label{pru2020} Let ${\bf x}= x_1, \ldots, x_n$ be  a joint reduction  of
$\mathbf{I}, J$ with respect to $M$
 of the type $({\bf
k},k_0+1).$  Assume that $k_i > 0$ and $x_1 \in I_i$ for $1 \le i \le d.$  Set ${\bf x'}= x_2, \ldots, x_n.$ Then
$\chi(\mathbf{x}, J, \mathbf{I}, M)=\chi(\mathbf{x'}, J, \mathbf{I}, M/ x_1M)-\chi(\mathbf{x'}, J, \mathbf{I}, 0_M : x_1).$
\end{proposition}
\begin{proof} Without loss of generality, we may assume that $k_1 > 0$ and $x_1 \in I_1.$
Set ${\bf x'}_T = {\bf x}_T \setminus \{x_1t_1\}.$ It is easily seen
that ${\bf x'}$ is a joint reduction of
$\mathbf{I}, J$ with respect to
$M/x_1M.$
Recall that $\mathbf{x}_T$ is  a joint reduction of $\mathbf{I}_T, J$
with respect to $\frak R(\mathbf{I}; M).$ Hence ${\bf x'}_T$ is a joint reduction of $\mathbf{I}_T, J$ with respect to
$\frak R(\mathrm{\bf I}; M)/x_1t_1 \frak R(\mathrm{\bf I}; M).$ Moreover, ${\bf x'}_T$ is a joint reduction of $\mathbf{I}_T, J$ with respect to
$(0_{\frak R(\mathrm{\bf I}; M)}:x_1t_1)$ by Note \ref{n2} (iv).
Since $(0_M: x_1)$ is an $A$-submodule of $M,$ by \cite[Corollary 2.8]{TV4}, ${\bf x}$  is also a joint reduction of
$\mathbf{I}, J$ with respect to $(0_M: x_1).$ Hence since $x_1(0_M: x_1) = 0,$ it follows that ${\bf x'}$ is a joint reduction of
$\mathbf{I}, J$ with respect to
$(0_M: x_1).$
Now, note that for any $a \in A,$ we have $$\{x \in \mathbb{ I}^{\mathrm{\bf n}}M |\; ax = 0_M \} = \mathbb{ I}^{\mathrm{\bf n}}M \cap (0_M:a).$$
So $0_{\frak R(\mathrm{\bf I}; M)} : x_1t_1=
\bigoplus_{\bf n \geqslant 0}[\mathbb{ I}^{\mathrm{\bf n}}M \cap (0_M : x_1)]T^{\bf n}.$
 By Artin-Rees Lemma, there exists
an integer $k>0$ such that
$\mathbb{I}^{\mathbf{n}+k\mathbf{1}}M\cap (0_M : x_1) =
\mathbb{I}^{\mathbf{n}}(I^{k}M\cap (0_M : x_1))$ for all ${\bf n \ge 0}.$
Fix this integer $k$ and set $M'= I^{k}M\cap (0_M : x_1).$
Then $$[0_{\frak R(\mathrm{\bf I}; M)} : x_1t_1]_{{\bf n} + k{\bf 1}} \cong_A [\frak R(\mathrm{\bf I}; M')]_{\bf n}$$ for all ${\bf n \ge 0}.$
On the other hand, recall that $\chi(n_0,\mathbf{n}, {\bf x'}_T, J, \frak R(\mathrm{\bf I}; M'))=\chi(\mathbf{x'}, J, \mathbf{I}, M')$ is a constant by (\ref{vv8/3}).
Hence
by \cite[Remark 3.4 (ii)]{TV4}, we have
\begin{equation}\label{vvvv}\chi(n_0,\mathbf{n}, {\bf x'}_T, J,0_{\frak R(\mathrm{\bf I}; M)} : x_1t_1) = \chi(n_0,\mathbf{n}, {\bf x'}_T, J, \frak R(\mathrm{\bf I}; M')).\end{equation}
Since $I^{k}(0_M : x_1) \subseteq M'= I^{k}M\cap (0_M : x_1) \subseteq (0_M : x_1),$ by Remark \ref{ree2020} (ii),
$$\chi(\mathbf{x'}, J, \mathbf{I}, I^{k}(0_M : x_1)) \le \chi(\mathbf{x'}, J, \mathbf{I}, M') \le \chi(\mathbf{x'}, J, \mathbf{I}, (0_M : x_1)).$$ And by Remark \ref{ree2020} (iii), $\chi(\mathbf{x'}, J, \mathbf{I}, I^{k}(0_M : x_1)= \chi(\mathbf{x'}, J, \mathbf{I}, (0_M : x_1)).$ Therefore
$$\chi(\mathbf{x'}, J, \mathbf{I}, M')= \chi(\mathbf{x'}, J, \mathbf{I}, (0_M : x_1)).$$ Hence
$\chi(n_0,\mathbf{n}, {\bf x'}_T, J, \frak R(\mathbf{I}; M')) = \chi(\mathbf{x'}, J, \mathbf{I}, (0_M : x_1)).$ So by (\ref{vvvv}),
\begin{equation}\label{v22} \chi(n_0,\mathbf{n}, {\bf x'}_T, J, 0_{\frak R(\mathrm{\bf I}; M)} : x_1t_1)=\chi(\mathbf{x'}, J, \mathbf{I}, (0_M : x_1)).
\end{equation}
 Next, we will show that
 \begin{equation}\label{vv2}\chi(n_0,\mathbf{n}, {\bf x'}_T, J,\frak R(\mathbf{I}; M)/x_1t_1\frak R(\mathrm{\bf I}; M)) = \chi(\mathbf{x'}, J, \mathbf{I}, M/ x_1M).\end{equation}
   Set
 $$E = x_1M\oplus[\bigoplus_{\bf n \ge \bf e_{1}}(\mathbb{ I}^{\mathrm{\bf n}}M\cap x_1M/\mathbb{ I}^{\mathrm{\bf n-e}_1}x_1M)T^{\bf n}];F= \bigoplus_{\bf n \ge \bf e_{1}}(\mathbb{ I}^{\mathrm{\bf n}}M\cap x_1M/\mathbb{ I}^{\mathrm{\bf n-e}_1}x_1M)T^{\bf n}.$$
      Note that
   $\frak R(\mathbf{I}; M/x_1M) \cong \bigoplus_{\bf n \ge \bf 0}(\mathbb{ I}^{\mathrm{\bf n}}M/\mathbb{ I}^{\mathrm{\bf n}}M\cap x_1M)T^{\bf n}$
   and $$\frak
R(\mathbf{I}; M)/x_1t_1\frak R(\mathrm{\bf I}; M) \cong M\oplus[\bigoplus_{\bf n \ge \bf e_{1}}(\mathbb{ I}^{\mathrm{\bf n}}M/\mathbb{ I}^{\mathrm{\bf n-e}_1}x_1M)T^{\bf n}].$$
So we have the following exact sequence of $\frak R(\mathbf{I}; A)$-modules:

\begin{equation}\label{vv}0\longrightarrow E \longrightarrow \frak
R(\mathbf{I}; M)/x_1t_1\frak R(\mathrm{\bf I}; M)\longrightarrow\frak R(\mathbf{I}; M/x_1M)\longrightarrow 0,\end{equation}
   Since ${\bf x'}$ is a joint reduction of
$\mathbf{I}, J$ with respect to
$M/x_1M,$ it follows that ${\bf x'}_T$ is a joint reduction of $\mathbf{I}_T, J$ with respect to
$\frak R(\mathrm{\bf I}; M/x_1M).$ And since ${\bf x'}_T$ is a joint reduction of $\mathbf{I}_T, J$ with respect to
$\frak R(\mathbf{I}; M)/x_1t_1\frak R(\mathrm{\bf I}; M),$ by Note \ref{n2} (iii), ${\bf x'}_T$ is a joint reduction of $\mathbf{I}_T, J$ with respect to $E.$

Recall that by Artin-Rees Lemma, there exists
an integer $u>0$ such that
$\mathbb{I}^{\mathbf{n}+u\mathbf{1}}M\cap  x_1M =
\mathbb{I}^{\mathbf{n}}(I^{u}M\cap x_1M)$ for all ${\bf n \ge 0}.$  Hence
$\mathbb{I}^{\mathbf{n}+u\mathbf{1}}M\cap  x_1M \subseteq
\mathbb{I}^{\mathbf{n}}x_1M$ for all ${\bf n \ge 0}.$ Note that
$$I^u[\mathbb{ I}^{\mathrm{\bf n}}M\cap x_1M] \subseteq \mathbb{I}^{\mathbf{n}+u\mathbf{1}}M\cap  x_1M$$ for all ${\bf n \ge 0}.$ Therefore
$I^u[\mathbb{ I}^{\mathrm{\bf n}}M\cap x_1M] \subseteq \mathbb{I}^{\mathbf{n}}x_1M$ for all ${\bf n \ge 0}.$
Consequently, $$I^uF = I^u[\bigoplus_{\bf n \ge \bf e_{1}}(\mathbb{ I}^{\mathrm{\bf n}}M\cap x_1M/\mathbb{ I}^{\mathrm{\bf n-e}_1}x_1M)T^{\bf n}] = 0.$$
So
$I \subset \sqrt{\mathrm{Ann}_{\frak R(\mathrm{\bf I}; A)}F}.$ Now, let
$0 = F_0 \subseteq F_1 \subseteq F_2 \subseteq\cdots \subseteq F_t=F$ be a homogeneous prime  filtration of $F,$ i.e., $F_{i+1}/F_i \cong \frak R(\mathrm{\bf I}; A)/P_i,$ where  $P_i \in \mathrm{Supp}F$ is a homogeneous prime ideal of $\frak R(\mathrm{\bf I}; A)$ for all $0 \le i \le t-1.$
Since ${\bf x'}_T$ is a joint reduction of $\mathbf{I}_T, J$ with respect to
$E,$ it follows by Note \ref{n2} (iii), that ${\bf x'}_T$ is a joint reduction of $\mathbf{I}_T, J$ with respect to $F.$ Consequently,
${\bf x'}_T$ is a joint reduction of $\mathbf{I}_T, J$ with respect to
$F_i$ for all $0 \le i \le t$ by  Note \ref{n2} (iii).
Hence ${\bf x'}_T$ is a joint reduction of $\mathbf{I}_T, J$ with respect to $\frak R(\mathrm{\bf I}; A)/P_i$ for all $0 \le i \le t-1.$
Then by (\ref{v021}), we have
\begin{equation}\label{26/7} \chi(n_0,\mathbf{n}, {\bf x'}_T, J, F) =\sum_{i=0}^{t-1}\chi(n_0,\mathbf{n}, {\bf x'}_T, J, \frak R(\mathrm{\bf I}; A)/P_i). \end{equation}
It is well known that $\mathrm{Min}F \subset \{P_0, \ldots, P_{t-1} \} \subset \mathrm{Supp}F.$
Now, let $\frak P \in \{P_0, \ldots, P_{t-1} \}.$ Then there exists $P\in \mathrm{Min}F$ such that
$P \subset \frak P.$ Since $P\in \mathrm{Min}F, $ it follows that  $P$ is an associated prime ideal of $F.$ Hence
 there exists a homogeneous element $x \in F$ such that $P = 0_F:x.$ Set $\frak p = P\cap A= \{a \in A \mid ax = 0_F\}.$
 Then it can be verified that
 $P = \bigoplus_{\bf n \geq 0}(\mathbb{ I}^{\mathrm{\bf n}}\cap \frak p)T^{\bf n}.$
Recall that  $I \subset \sqrt{\mathrm{Ann}_{\frak R(\mathrm{\bf I}; A)}F}$ and $P\in \mathrm{Supp}F.$ Hence $I \subset  P.$ So $I \subset  \frak p.$
From this it follows that $\bigoplus_{\bf n \geq 1}\mathbb{ I}^{\mathrm{\bf n}}T^{\bf n} \subseteq P.$ Consequently, $\bigoplus_{\bf n \geq 1}\mathbb{ I}^{\mathrm{\bf n}}T^{\bf n} \subseteq \frak P.$ Therefore, $(\frak R(\mathrm{\bf I}; A)/\frak P)_{{\bf n} + {\bf 1}} = 0$ for all $\bf n.$
And in this case, we get
$\chi(n_0,\mathbf{n}, {\bf x'}_T, J, \frak R(\mathrm{\bf I}; A)/\frak P)=0$ by \cite[Remark 3.4 (ii)]{TV4}.
So
$\chi(n_0,\mathbf{n}, {\bf x'}_T, J, \frak R(\mathrm{\bf I}; A)/P_i) = 0$ for all $0 \le i \le t-1.$ Hence $\chi(n_0,\mathbf{n}, {\bf x'}_T, J, F) =0$ by  (\ref{26/7}). Next, since $E_{\bf n}=F_{\bf n}$ for all $\bf n \ge \bf e_{1},$ by \cite[Remark 3.4 (ii)]{TV4}, we have
$\chi(n_0,\mathbf{n},\mathbf{x'}_T, J, E) = \chi(n_0,\mathbf{n},\mathbf{x'}_T, J, F).$ Thus,
$\chi(n_0,\mathbf{n},\mathbf{x'}_T, J, E) = 0.$
Hence from (\ref{vv}), we obtain by (\ref{v021}), that
 $$\chi(n_0,\mathbf{n}, {\bf x'}_T, J,\frak R(\mathbf{I}; M)/x_1t_1\frak R(\mathrm{\bf I}; M)) = \chi(n_0,\mathbf{n}, {\bf x'}_T, J, \frak R(\mathbf{I}; M/x_1M)).$$
 Recall that $\chi(n_0,\mathbf{n}, {\bf x'}_T, J, \frak R(\mathbf{I}; M/x_1M)) = \chi(\mathbf{x'}, J, \mathbf{I}, M/ x_1M)$
   by (\ref{vv8/3}). So we get (\ref{vv2}):
 $\chi(n_0,\mathbf{n}, {\bf x'}_T, J,\frak R(\mathbf{I}; M)/x_1t_1\frak R(\mathrm{\bf I}; M)) = \chi(\mathbf{x'}, J, \mathbf{I}, M/ x_1M).$
 Combining (\ref{vv8/3}), (\ref{v22}) and (\ref{vv2}) yields, by Proposition \ref{lem 3.4} (iii), that
  $$\chi(\mathbf{x}, J, \mathbf{I}, M)=\chi(\mathbf{x'}, J, \mathbf{I}, M/ x_1M)-\chi(\mathbf{x'}, J, \mathbf{I}, (0_M : x_1)).$$
   \end{proof}

\section{Mixed multiplicities of ideals}

In this section, we prove the recursion formula for mixed multiplicities of ideals, and
using this recursion formula, we give results on the relationship between mixed multiplicities and the Hilbert-Samuel multiplicity of joint reductions.

Recall that $I = I_1\cdots I_d;$
  $\overline {M}= M/0_M: I^\infty;$ $q=\dim \overline {M}$; $\mathrm{\bf I}^{[\mathrm{\bf
k}]}= I_1^{[k_1]},
  \ldots,I_d^{[k_d]}.$
Denote by $P(n_0, {\bf n}, J, \mathbf{I}, M)$ the Hilbert polynomial of the function
 $\ell_A\Big(\dfrac{J^{n_0}\mathbb{I}^{\bf
n}M}{J^{n_0+1}\mathbb{I}^{\bf n}M}\Big).$ Assign the dimension $-\infty$ to the module $0$ and
the degree $-\infty$ to the zero polynomial. Note that by Remark \ref{no4.3a} (iii), $\dim \overline {M} \not = 0.$
Then by \cite[Proposition
3.1]{Vi} (see \cite{MV}),
  $P(n_0, {\bf n}, J,  \mathbf{I}, M)$
 is a  polynomial of total
degree $q-1.$ One  can
write
$$P(n_0, {\bf n}, J, \mathbf{I}, M) = \sum_{(k_0, \;\mathbf{k})\in \mathbb{N}^{d+1}
}e(J^{[k_0+1]}, \mathbf{I}^{[\mathbf{k}]}; M)\binom{n_0 +
k_0}{k_0}\binom{\mathbf{n}+\mathbf{k}}{\mathbf{k}},$$
 here
$\binom{\mathbf{n + k}}{\bf n}= \binom{n_1 + k_1}{n_1}\cdots
\binom{n_d + k_d}{n_d}.$
 It is clear that
$\bigtriangleup^{(k_0,\; \mathrm{\bf
k})}P(n_0, {\bf n}, J,  \mathbf{I}, M)$ is a constant if and only if
 $e(J^{[h_0+1]},
\mathbf{I}^{[\mathbf{h}]}; M)=0$ for all $(h_0, \mathbf{h}) >
(k_0,\mathbf{k}).$ In this case,
\begin{equation}\label{vv25}\bigtriangleup^{(k_0, \;\mathrm{\bf
k})}P(n_0, {\bf n}, J,  \mathbf{I}, M)= e(J^{[k_0+1]}, \mathbf{I}^{[\mathbf{k}]}; M) \end{equation} (see \cite[Proposition 2.4 (ii)]{htv}).
And \cite{htv} gave a larger class than the class of original
mixed multiplicities, that $e(J^{[k_0+1]},
\mathbf{I}^{[\mathbf{k}]}; M)$ is  the {\it  mixed
multiplicity of
 maximal degrees of $M$ with respect to ideals  $J,\mathrm{\bf I}$ of the type
$(k_0+1,\mathrm{\bf k})$} if $e(J^{[h_0+1]},
\mathbf{I}^{[\mathbf{h}]}; M)=0$ for all $(h_0, \mathbf{h}) >
(k_0,\mathbf{k}).$ By \cite[Proposition 2.7]{TV4} (see \cite[Remark 4.3]{TV4}),
 the mixed
multiplicity of  maximal degrees
 $e(J^{[k_0+1]},
\mathbf{I}^{[\mathbf{k}]}; M)$
 is defined if and only if
there exists  a joint
reduction of
$\mathbf{I}, J$ with respect to $M$
 of the type $({\bf
k},k_0+1),$
 and then  we have
\begin{equation}\label{vv221}e(J^{[k_0+1]},
\mathbf{I}^{[\mathbf{k}]}; M) \ge 0 \end{equation}
by \cite[Proposition 2.4 (i)]{htv} (see \cite[Note 2.10]{TV4}).

And as a translation of Proposition \ref{pru2020} into the language of the mixed multiplicity, we get the following recursion formula for mixed multiplicities.

\begin{theorem} \label{le2020} Let ${\bf x}= x_1, \ldots, x_n$ be  a joint reduction  of
$\mathbf{I}, J$ with respect to $M$
 of the type $({\bf
k},k_0+1).$  Assume that $k_i > 0$ and $x_1 \in I_i$ for $1 \le i \le d.$ Then
$$e(J^{[k_0+1]}, \mathbf{I}^{[\mathbf{k}]}; M) =
e(J^{[k_0+1]}, \mathbf{I}^{[\mathrm{\bf k} - \mathbf{e}_i]};
 M/x_1M) - e(J^{[k_0+1]}, \mathbf{I}^{[\mathrm{\bf k} - \mathbf{e}_i]};
 0_M: x_1).$$
\end{theorem}
\begin{proof} Set $\mathbf{x'} = x_2, \ldots, x_n.$ Since ${\bf x}$ is  a joint reduction  of
$\mathbf{I}, J$ with respect to $M$
 of the type $({\bf
k},k_0+1),$ it is easily seen
 that $\mathbf{x'}$ is a joint reduction  of
$\mathbf{I}, J$ with respect to $M/x_1M$ of the type $({\bf
k-e}_i,k_0+1).$ Since $(0_M: x_1)$ is an $A$-submodule of $M,$ by \cite[Corollary 2.8]{TV4},
 ${\bf x}$  is a joint reduction of
$\mathbf{I}, J$ with respect to $(0_M: x_1).$  Therefore, since $x_1(0_M: x_1) = 0,$ it follows that ${\bf x'}$ is a joint reduction of
$\mathbf{I}, J$ with respect to
$(0_M: x_1)$ of the type $({\bf
k-e}_i,k_0+1).$
 Hence  by \cite[Theorem 4.4 (i)]{TV4},
\begin{align*} \chi(\mathbf{x}, J, \mathbf{I}, M)&= e(J^{[k_0+1]}, \mathbf{I}^{[\mathbf{k}]}; M);\\
\chi(\mathbf{x'}, J, \mathbf{I}, M/ x_1M) &=e(J^{[k_0+1]}, \mathbf{I}^{[\mathrm{\bf k} - \mathbf{e}_i]}; M/x_1M);\\
\chi(\mathbf{x'}, J, \mathbf{I}, 0_M : x_1)&=e(J^{[k_0+1]}, \mathbf{I}^{[\mathrm{\bf k} - \mathbf{e}_i]};  0_M: x_1).
\end{align*}
Consequently, by Proposition \ref{pru2020}, we immediately get the proof of Theorem \ref{le2020}.
\end{proof}

From the Theorem \ref{le2020}, we easily see the following consequences.

\begin{corollary}\label{vt5/3} Let ${\bf x}= x_1, \ldots, x_n$ be  a joint reduction  of
$\mathbf{I}, J$ with respect to $M$
 of the type $({\bf
k},k_0+1).$  Assume that $k_i > 0$ and $x_1 \in I_i$ for $1 \le i \le d.$ Then we have
\begin{itemize}
\item[$\mathrm{(i)}$]
$e(J^{[k_0+1]}, \mathbf{I}^{[\mathbf{k}]}; M)\le
e(J^{[k_0+1]}, \mathbf{I}^{[\mathrm{\bf k} - \mathbf{e}_i]};
 M/x_1M).$
  \item[$\mathrm{(ii)}$] $e(J^{[k_0+1]}, \mathbf{I}^{[\mathbf{k}]}; M)=
e(J^{[k_0+1]}, \mathbf{I}^{[\mathrm{\bf k} - \mathbf{e}_i]};
 M/x_1M)$  if $x_1$ is an $M$-regular element.
 \item[$\mathrm{(iii)}$] $e(J^{[k_0+1]}, \mathbf{I}^{[\mathbf{k}]}; M)=
e(J^{[k_0+1]}, \mathbf{I}^{[\mathrm{\bf k} - \mathbf{e}_i]};
 M/x_1M)$ if $x_1$ is an $I$-filter-regular element with respect to $M.$
\end{itemize}
\end{corollary}
\begin{proof} Since $e(J^{[k_0+1]}, \mathbf{I}^{[\mathrm{\bf k} - \mathbf{e}_i]};
 0_M: x_1) \ge 0$ by (\ref{vv221}), we have (i)
  by Theorem \ref{le2020}. Now, if $x_1$ is an $M$-regular element, then $(0_M : x_1) = 0_M.$ Hence
  $$e(J^{[k_0+1]}, \mathbf{I}^{[\mathrm{\bf k} - \mathbf{e}_i]};
 0_M: x_1) = 0.$$ So by Theorem \ref{le2020}, we get (ii).
 If $x_1$ is $I$-filter-regular with respect to $M,$ then
$0_M:x_1 \subseteq 0_M: I^{\infty}.$ Since $M$ is Noetherian, it follows that there exists
an integer $s>0$ such that $I^s[0_M:x_1] = 0.$
From this it follows that the Hilbert polynomial $$P(n_0, {\bf n}, J, \mathbf{I}, [0_M:x_1]) = 0.$$
 Therefore, $e(J^{[k_0+1]}, \mathbf{I}^{[\mathrm{\bf k} - \mathbf{e}_i]};
 0_M: x_1) = 0.$ Thus, we obtain (iii)
   by Theorem \ref{le2020}.
 \end{proof}

In order to prove the next results, we need the
following facts.

Recall that an ideal $\frak{a}$ of $A$ is called an {\it ideal of definition} of $M$ if $\ell_A(M/\frak{a}M)<\infty,$
and a sequence ${\bf y} = y_1, \ldots, y_n$ of elements
in $ \frak m$ is called a {\it multiplicity system} of $M$ if
$(\bf y)$ is an ideal of definition of $M.$ Let ${\bf y}$ be a
multiplicity system. Then one defined the {\it multiplicity symbol}
of ${\bf y}$ as follows: if $n=0$, then $\ell_A(M) <\infty$, and
set $e(\mathbf{y};M) = \ell_A(M)$. If $n> 0$, set $e(\mathbf{y};
M) = e(\mathbf{y}'; M/y_1M) - e(\mathbf{y}'; 0_{M}:y_1),$ here $\mathbf{y}'= y_2, \ldots, y_n.$
It
is well known that $e(\mathbf{y}; M) \not=0$ if and only if   $\bf
y$ is a system of parameters for $M$, and in this case,
$e(\mathbf{y};  M) = e((\mathbf{y});  M)$ (see e.g.
\cite{BH1}).

\begin{remark}\label{no4.3a} With the previous notions $\mathbf{I},J$ and assume that
  $\bf x$ is a joint
reduction of  $\mathbf{I},J$ with respect to $M$ of the type $(
{\bf 0}, k_0+1).$  Then we have the following.
\begin{enumerate}[{\rm (i)}]
\item Since $\bf x$ is a joint
reduction of the type $( {\bf 0}, k_0+1),$ it follows that $(\bf x)$ is a
reduction of $J$ with respect to $I^mM$ for large enough $m.$ Recall that $J$ is $\frak m$-primary.
So $({\bf x})$ is an ideal of definition of $I^mM$ for large enough $m.$ Note that
$\sqrt{\mathrm{Ann}[\overline
{M}/(\mathbf{x})\overline {M}}]=
\sqrt{\mathrm{Ann}[I^mM/(\mathbf{x})I^mM]}$ for all large
enough  $m.$ Hence $({\bf x})$ is an ideal of definition of
$\overline {M}.$
 \item  By (i), $\mathbf{x}$ is a
multiplicity system of $\overline{M}.$ So $\dim \overline{M} \le |\mathbf{x}| = k_0+1.$
Then if $\dim \overline{M} = k_0+1,$  then
${\bf x}$ is a system of parameters for $ \overline{M}.$ Hence by \cite[Lemma
3.2 (i)]{Vi},
$e(J^{[k_0+1]},
\mathbf{I}^{[\mathbf{0}]}; M) = e(J;
\overline{M}) =  e(\mathbf{x};
\overline{M}).$
If $\dim \overline{M} < k_0+1,$ then $\deg P(n_0, {\bf n}, J,  \mathbf{I}, M) < k_0.$ Hence
$e(J^{[k_0+1]},
\mathbf{I}^{[\mathbf{0}]}; M) = 0$ by (\ref{vv25}), and ${\bf x}$ is not a system of parameters for $ \overline{M}.$
In this case, $e(\mathbf{x};
\overline{M}) =0.$
 So we always  get
$e(J^{[k_0+1]},
\mathbf{I}^{[\mathbf{0}]}; M) = e(\mathbf{x};
\overline{M}).$ This can also be seen in \cite[Lemma 2.6 (ii)]{TV3}.
  \item Since $0_M: I^\infty = [0_M: I^\infty]:I,$ it follows that if $\overline{M} \not =0,$ then  $\dim \overline{M}> 0.$ Hence
    $\dim \overline{M} \le 0$ if and only if $\overline{M} =0,$ i.e., $\dim \overline{M} = -\infty.$
        This is equivalent to $I \subseteq \sqrt{\mathrm{Ann}M}$ (see \cite[Lemma 2.9 (i)]{TV4}). So we always have $\dim \overline{M} \not =0.$

   \end{enumerate}
\end{remark}

 To illustrate further the effectiveness of Theorem \ref{le2020}, we consider the case of ${\frak m}$-primary ideals by the following note.

\begin{note}\label{note.1} Let $I_1, \ldots, I_d, J$ be ${\frak m}$-primary ideals.
Let ${\bf x}= x_1, \ldots, x_n$ be  a joint reduction  of
$\mathbf{I}, J$ with respect to $M$
 of the type $({\bf
k},k_0+1).$ In this case, $I$ is an ${\frak m}$-primary ideal. Then
it is easy to check that ${\bf x}$ is a
multiplicity system of $\overline {M}$ and $M$ and  $e({\bf x}; M)= e({\bf x}; \overline {M}).$ Now, using Theorem \ref{le2020}, we prove by induction on $|{\bf k}|$ that $e(J^{[k_0+1]}, \mathbf{I}^{[\mathbf{k}]}; M) =e({\bf x}; M).$ Indeed, if $|{\bf k}|=0,$ then ${\bf x}$ is  a joint reduction  of
$\mathbf{I}, J$ with respect to $M$
 of the type $({\bf 0}, k_0+1).$ By Remark \ref{no4.3a} (ii), $e(J^{[k_0+1]}, \mathbf{I}^{[\mathbf{0}]}; M) =e({\bf x}; \overline {M}).$ So
$e(J^{[k_0+1]}, \mathbf{I}^{[\mathbf{0}]}; M) =e({\bf x}; {M}).$
 Consider the case that $|{\bf k}|> 0.$ And without loss of generality, we can assume that
$k_1>0$ and $x_1\in I_1.$ By Theorem \ref{le2020}, $e(J^{[k_0+1]}, \mathbf{I}^{[\mathbf{k}]}; M) =
e(J^{[k_0+1]}, \mathbf{I}^{[\mathrm{\bf k} - \mathbf{e}_1]};
 M/x_1M) - e(J^{[k_0+1]}, \mathbf{I}^{[\mathrm{\bf k} - \mathbf{e}_1]};
 0_M: x_1).$ Since ${\bf x'}= x_2, \ldots, x_n$ is  a joint reduction  of
$\mathbf{I}, J$ with respect to $M/x_1M$ and $0_M: x_1$
 of the type $({\bf
k-e}_1,k_0+1),$ by the
inductive hypothesis, we have \begin{align*}e(J^{[k_0+1]}, \mathbf{I}^{[\mathrm{\bf k} - \mathbf{e}_1]};
 M/x_1M) &= e({\bf x'}; {M/x_1M});\\ e(J^{[k_0+1]}, \mathbf{I}^{[\mathrm{\bf k} - \mathbf{e}_1]};
 0_M: x_1) &= e({\bf x'}; 0_M: x_1).\end{align*} Recall that $e({\bf x}; M)= e({\bf x'}; {M/x_1M})-e({\bf x'}; 0_M: x_1).$
  So
  $e(J^{[k_0+1]}, \mathbf{I}^{[\mathbf{k}]}; M) =  e({\bf x}; M).$
     Therefore, we recover the Rees's mixed multiplicity theorem in \cite{Re}.
\end{note}

In addition, as an immediate consequence of Corollary \ref {vt5/3} (i) and (iii), we also obtain the following results.

\begin{corollary}\label{vt15/7}  Let ${\bf x}= x_1, \ldots, x_n$ be  a joint reduction  of
$\mathbf{I}, J$ with respect to $M$
 of the type $({\bf
k},k_0+1)$ with
${\bf x}_{\mathbf{I}} =x_1,\ldots,x_{|{\bf k }|} \subset {\bf I}$ and $U =x_{|{\bf k }|+1},\ldots,x_n \subset J.$
Set $I= I_1\cdots I_d.$ Then
\begin{itemize}
\item[$\mathrm{(i)}$] $e(J^{[k_0 +1]}, \mathbf{I}^{[\mathbf{k}]}; M)
\le e(U; {M}/({\bf x}_{\mathbf{I}})M:I^\infty).$
\item[$\mathrm{(ii)}$] $e(J^{[k_0 +1]}, \mathbf{I}^{[\mathbf{k}]}; M)
= e(U; {M}/({\bf x}_{\mathbf{I}})M:I^\infty)$ if ${\bf x}_{\mathbf{I}}$ is an $I$-filter-regular sequence with respect to $M.$
\end{itemize}
\end{corollary}
\begin{proof}
  Since  ${\bf x}$ is a joint reduction of the type $({\bf
k},k_0+1)$ of
$\mathbf{I}, J$ with respect to $M,$
 it follows that $U$ is a joint reduction  of
$\mathbf{I}, J$ with respect to $M/({\bf x}_{\mathbf{I}})M$ of the type $({\bf
0},k_0+1).$
  Hence $e(J^{[k_0 +1]}, \mathbf{I}^{[\mathbf{0}]}; M/({\bf x}_{\mathbf{I}})M) = e(U; {M}/({\bf x}_{\mathbf{I}})M:I^\infty)$ by
Remark \ref{no4.3a} (ii). Now, by Corollary \ref {vt5/3} (i), $e(J^{[k_0 +1]}, \mathbf{I}^{[\mathbf{k}]}; M) \le e(J^{[k_0 +1]}, \mathbf{I}^{[\mathbf{0}]}; M/({\bf x}_{\mathbf{I}})M).$ So we have (i). If ${\bf x}_{\mathbf{I}}$ is an $I$-filter-regular sequence with respect to $M,$ then
by Corollary \ref {vt5/3} (iii), $e(J^{[k_0 +1]}, \mathbf{I}^{[\mathbf{k}]}; M = e(J^{[k_0 +1]}, \mathbf{I}^{[\mathbf{0}]}; M/({\bf x}_{\mathbf{I}})M).$ Therefore we get (ii).
\end{proof}

Now, we prove following result on the relationship between mixed multiplicities and the Hilbert-Samuel multiplicity of joint reductions.

 \begin{corollary}\label{vt2019}  Let ${\bf x}= x_1, \ldots, x_n$ be  a joint reduction  of
$\mathbf{I}, J$ with respect to $M$
 of the type $({\bf
k},k_0+1)$ with
${\bf x}_{\mathbf{I}} =x_1,\ldots,x_{|{\bf k }|} \subset {\bf I}.$ Assume that ${\bf x}$ is a system of parameters for ${M}.$
Set $I= I_1\cdots I_d.$ Then
\begin{itemize}
\item[$\mathrm{(i)}$] $e(J^{[k_0 +1]}, \mathbf{I}^{[\mathbf{k}]}; M)
\le e(\mathbf{x}; {M}).$
\item[$\mathrm{(ii)}$] $e(J^{[k_0 +1]}, \mathbf{I}^{[\mathbf{k}]}; M)
= e(\mathbf{x}; {M})$ if
 $\dim {M}/({\bf x}_{\mathbf{I}},I){M} < \dim
{M}/({\bf x}_{\mathbf{I}}){M}.$
\end{itemize}
\end{corollary}
\begin{proof} We prove this theorem  by induction
 on $|{\bf k}|.$
If $|{\bf k}| = 0$, then since $\bf x$ is a joint reduction of $
\mathbf{I}, J$ with respect to $M$ of the type $({\bf 0}, k_0+1)$,
by Remark \ref{no4.3a} (ii) we have
  $e(J^{[k_0+1]},
\mathbf{I}^{[\mathbf{0}]}; M) = e(\mathbf{x};{M}/0_M:I^\infty).$ Hence since $e(\mathbf{x};{M}/0_M:I^\infty) \le e(\mathbf{x};M),$ we obtain
$e(J^{[k_0+1]},\mathrm{\bf I}^{[\mathrm{\bf 0}]}; M) \le e(\mathbf{x}; {M}).$
And if $\dim M/IM < \dim M,$ then it is clear that $\dim (0_M:I^\infty) < \dim M.$
Then since the exact sequence of $A$-modules:
$$0\longrightarrow (0_M:I^\infty) \longrightarrow M \longrightarrow M/0_M:I^\infty\longrightarrow 0,$$ it follows that
$e(\mathbf{x};{M}/0_M:I^\infty) = e(\mathbf{x};M).$ Thus, $e(J^{[k_0+1]},\mathrm{\bf I}^{[\mathrm{\bf 0}]}; M) = e(\mathbf{x}; {M}).$
 We get the proof in the case $|\mathbf{k}| = 0.$
 Consider the case that $|\mathbf{k}| >0.$ And without loss of generality, assume that
$k_1>0$ and $x_1\in I_1.$
Set $$\Pi = \{\frak p \in \mathrm{Min}{M}\mid \dim A/\frak p = \dim
{M}\}.$$ By \cite[Corollary 4.6]{TV4} (see \cite[Theorem 3.2]{VT1}), we have
$$e(J^{[k_0+1]},\mathrm{\bf I}^{[\mathrm{\bf k}]}; M)
=\sum_{\mathfrak{p}\in
\Pi}\ell_A(M_{\mathfrak{p}})e(J^{[k_0+1]},\mathrm{\bf
I}^{[\mathrm{\bf k}]}; A/\mathfrak{p}).$$
Note that since ${\bf x}= x_1, \ldots, x_n$ is  a joint reduction  of
$\mathbf{I}, J$ with respect to $M$
 of the type $({\bf
k},k_0+1),$ it follows that ${\bf x}$ is also a joint reduction of $\mathbf{I}, J$ with respect to $A/\frak p$ of the type $({\bf
k},k_0+1)$ for all
$\mathfrak{p}\in
\Pi$ (see e.g \cite[Lemma 17.1.4]{SH}). Therefore $x_2,\ldots,x_n$ is a joint reduction of $\mathbf{I}, J$ with respect to $A/(x_1,\mathfrak{p})$
 of the type $({\bf
k-e}_1,k_0+1)$ for all
$\mathfrak{p}\in
\Pi.$ And since ${\bf x}$ is a system of parameters for ${M},$ it follows that  ${\bf x}$ is a system of parameters for $A/\frak p$ for all
$\mathfrak{p}\in
\Pi.$ Hence $x_2,\ldots,x_n$ is a system of parameters for $A/(x_1,\mathfrak{p})$ for all
$\mathfrak{p}\in
\Pi.$ Now since
 $x_1$ is an $A/\frak p$-regular element for all
$\mathfrak{p}\in
\Pi,$ it follows by Corollary \ref {vt5/3} (ii),
 that
  $$e(J^{[k_0+1]},\mathrm{\bf
I}^{[\mathrm{\bf k}]}; A/\mathfrak{p})=e(J^{[k_0+1]},
\mathbf{I}^{[\mathbf{k-e}_1]}; A/(x_1,\mathfrak{p}))$$ for all
$\mathfrak{p}\in \Pi.$
Then by the
inductive hypothesis,
$$e(J^{[k_0+1]},
\mathbf{I}^{[\mathbf{k-e}_1]}; A/(x_1,\mathfrak{p}))
 \le e(x_2,\ldots,x_n;A/(x_1,\mathfrak{p}))$$
 for all
$\mathfrak{p}\in \Pi.$
Now, since $x_1$ is not  a  zero divisor of $A/\mathfrak{p},$ it follows that $$e(x_2,\ldots,x_n; A/(x_1,\mathfrak{p})) = e(x_1,x_2,\ldots,x_n; A/\mathfrak{p})$$
for all
$\mathfrak{p}\in \Pi$ by \cite[Page 641, lines 27-28, (D)]{AB}. So $$e(J^{[k_0+1]},\mathrm{\bf
I}^{[\mathrm{\bf k}]}; A/\mathfrak{p}) \le e(x_1,x_2,\ldots,x_n; A/\mathfrak{p})$$ for all
$\mathfrak{p}\in \Pi.$ Hence
$$\sum_{\mathfrak{p}\in
\Pi}\ell_A(M_{\mathfrak{p}})e(J^{[k_0+1]},\mathrm{\bf
I}^{[\mathrm{\bf k}]}; A/\mathfrak{p}) \le \sum_{\mathfrak{p}\in
\Pi}\ell_A(M_{\mathfrak{p}})e(x_1,x_2,\ldots,x_n; A/\mathfrak{p}).$$
Recall that $\sum_{\mathfrak{p}\in
\Pi}\ell_A(M_{\mathfrak{p}})e(x_1,x_2,\ldots,x_n; A/\mathfrak{p}) = e(x_1,x_2,\ldots,x_n; M)$
(see e.g.\cite[Theorem 11.2.4]{SH}). Thus, $e(J^{[k_0 +1]}, \mathbf{I}^{[\mathbf{k}]}; M)
\le e(\mathbf{x}; {M}).$

Next, if $\dim {M}/({\bf x}_{\mathbf{I}},I){M} < \dim {M}/({\bf x}_{\mathbf{I}}){M},$ then there exists $b\in I$ such that $$\dim {M}/({\bf x}_{\mathbf{I}}, b){M} = \dim {M}/({\bf x}_{\mathbf{I}}){M}-1.$$ Hence since ${\bf x}_{\mathbf{I}}$ is part of a system of parameters for ${M},$ it follows that ${\bf x}_{\mathbf{I}},b$ is part of a system of parameters for ${M}.$  Consequently, ${\bf x}_{\mathbf{I}},b$ is part of a system of parameters for $A/\frak p$ for all
$\mathfrak{p}\in \Pi.$ Set $B= A/(x_1,\mathfrak{p}).$ Then $x_2,\ldots,x_{|{\bf k }|},b$ is part of a system of parameters for $B.$
And we get $\dim B/(x_2,\ldots,x_{|{\bf k }|},b)B = \dim B - |\bf k|.$ So
$$\dim B/(x_2,\ldots,x_{|{\bf k }|},I)B < \dim B - |{\bf k-e}_1| = \dim B/(x_2,\ldots,x_{|{\bf k }|})B \mathrm{,i.e.,}$$
$$\dim \dfrac{[A/(x_1,\frak p)]}{(x_2,\ldots,x_{|{\bf k }|}, I)[A/(x_1,\frak p)]}
 < \dim \dfrac{[A/(x_1,\frak p)]}{(x_2,\ldots,x_{|{\bf k }|})[A/(x_1,\frak p)]}$$
 for all $\mathfrak{p}\in \Pi.$ Hence by the
inductive hypothesis,
$$e(J^{[k_0+1]},
\mathbf{I}^{[\mathbf{k-e}_1]}; A/(x_1,\mathfrak{p}))
 = e(x_2,\ldots,x_n;A/(x_1,\mathfrak{p}))$$
 for all
$\mathfrak{p}\in \Pi.$ Consequently,
$e(J^{[k_0+1]},\mathrm{\bf
I}^{[\mathrm{\bf k}]}; A/\mathfrak{p}) = e(x_1,x_2,\ldots,x_n; A/\mathfrak{p})$ for all
$\mathfrak{p}\in \Pi.$ Thus,  $e(J^{[k_0 +1]}, \mathbf{I}^{[\mathbf{k}]}; M)
= e(\mathbf{x}; {M}).$
 \end{proof}

\begin{remark}\label{no4.3b} Let ${\bf x}= x_1, \ldots, x_n$ be  a joint reduction  of
$\mathbf{I}, J$ with respect to $M$
 of the type $({\bf k},k_0+1).$ \cite[Lemma 3.2]{TV3} has stated that if $\dim M/IM < \dim M -|\mathbf{k}|,$ then
 ${\bf x}$ is a system of parameters for ${M}.$ However, the argument in the proof of this lemma in \cite{TV3} is incorrect, but up to
 the present we have not overcome yet.
 Fortunately, \cite[Theorem 3.3]{TV3} still holds if
 further assumption that the joint reduction is a system of parameters for $M.$
 Therefore, Corollary \ref{vt2019} (ii) covers also \cite[Theorem 3.3]{TV3} because that if $\dim {M}/I{M} < \dim {M} -|{\bf k}|,$ then $\dim {M}/({\bf x}_{\mathbf{I}},I){M} < \dim
{M}/({\bf x}_{\mathbf{I}}){M}.$
\end{remark}

Returning now to  \cite[Theorem 3.1]{VDT},
this theorem is an extension of  Rees's mixed multiplicity theorem to
the case that the ideal $I$ have
height larger than $|\mathbf{k}|$ and the joint reduction ${\bf x}$ is a system of parameters for ${M}.$
Developing this result is one of motivations leads us to the following result.

\begin{corollary}\label{cr2019}  Let ${\bf x}= x_1, \ldots, x_n$ be  a joint reduction  of
$\mathbf{I}, J$ with respect to $M$
 of the type $({\bf
k},k_0+1)$ with ${\bf x}_{\mathbf{I}} =x_1,\ldots,x_{|{\bf k }|} \subset {\bf I}.$ Assume that $\mathrm{ht}\frac{\mathrm{Ann}[M/({\bf x}_{\mathbf{I}})M]+I}{\mathrm{Ann}[M/({\bf x}_{\mathbf{I}})M]} > 0$ and $k_0 + |{\bf k}| = \dim \overline{M}-1.$
Then
${\bf x}$ is a system of parameters for ${M}$ and $$e(J^{[k_0 +1]}, \mathbf{I}^{[\mathbf{k}]}; M)
= e(\mathbf{x}; {M}).$$
\end{corollary}
\begin{proof} First, we prove that ${\bf x}$ is a system of parameters for ${M}.$
Set $U = x_{|\mathbf{k}|+1},\ldots,x_n.$ Since  ${\bf x}$ is a joint reduction  of
$\mathbf{I}, J$ with respect to $M$
 of the type $({\bf
k},k_0+1),$ it follows that $U$ is a joint reduction  of
$\mathbf{I}, J$ with respect to $M/({\bf x}_{\mathbf{I}})M$
 of the type $({\bf
0},k_0+1).$ Hence by Remark \ref{no4.3a} (i), $(U)$ is an ideal of
definition of $M/({\bf x}_{\mathbf{I}})M:I^\infty.$
Since $$\mathrm{ht}\frac{\mathrm{Ann}[M/({\bf x}_{\mathbf{I}})M]+I}{\mathrm{Ann}[M/({\bf x}_{\mathbf{I}})M]} > 0,$$ it follows that $I \nsubseteq \frak p$ for any $\frak p \in \mathrm{Min}[{M}/({\bf x}_{\mathbf{I}}){M})].$ Consequently  we get
   \begin{equation}\label{vv21}\mathrm{Min}[{M}/({\bf x}_{\mathbf{I}}){M}]=\mathrm{Min}[{M}/({\bf x}_{\mathbf{I}}){M}:I^\infty]. \end{equation}
 So $\dim
{M}/({\bf x}_{\mathbf{I}}){M}=\dim {M}/({\bf x}_{\mathbf{I}}){M}:I^\infty.$
    Note that since  $(U)$ is an ideal of
definition of $M/({\bf x}_{\mathbf{I}})M:I^\infty,$ we get
        $$\dim {M}/({\bf x}_{\mathbf{I}}){M}:I^\infty \le |U| = k_0+1= \dim \overline{M} - |{\bf k}| .$$   Therefore, we obtain
$$\dim {M} - |{\bf k}| \le \dim
{M}/({\bf x}_{\mathbf{I}}){M}=\dim {M}/({\bf x}_{\mathbf{I}}){M}:I^\infty \le \dim \overline{M} - |{\bf k}|\le \dim {M} - |\bf k|.
$$ Thus  $\dim
{M}/({\bf x}_{\mathbf{I}}){M}=\dim {M} - |\bf k|,$ and then
  ${\bf x}_{\mathbf{I}}$ is part of a system of parameters for ${M}.$ Now, since
    $(U)$ is an ideal of
definition of $M/({\bf x}_{\mathbf{I}})M:I^\infty,$
 we have
   $\mathrm{Min}[{M}/({\bf x}_{\mathbf{I}}){M}:I^\infty + (U)M] = \{\frak m \}.$
      Note that any $ P \in \mathrm{Min}[{M}/({\bf x}_{\mathbf{I}}, U){M}],$ there exists $\frak p \in \mathrm{Min}[{M}/({\bf x}_{\mathbf{I}}){M}]$ such that $\frak p \subset P.$ And by (\ref{vv21}), $\frak p \in \mathrm{Min}[{M}/({\bf x}_{\mathbf{I}}){M}:I^\infty].$ Hence
  $P \in \mathrm{Min}[{M}/({\bf x}_{\mathbf{I}}){M}:I^\infty + (U)M].$ Consequently, $P = \frak m.$
  Therefore, ${\bf x}= {\bf x}_{\mathbf{I}}, U $ is a system of parameters for ${M}.$

  Next, since $\mathrm{ht}\frac{\mathrm{Ann}[M/({\bf x}_{\mathbf{I}})M]+I}{\mathrm{Ann}[M/({\bf x}_{\mathbf{I}})M]} > 0,$ we have
 $$\dim {M}/({\bf x}_{\mathbf{I}},I){M} < \dim
{M}/({\bf x}_{\mathbf{I}}){M}.$$ Hence $e(J^{[k_0 +1]}, \mathbf{I}^{[\mathbf{k}]}; M)
= e(\mathbf{x}; {M})$ by Corollary \ref{vt2019} (ii).
\end{proof}

Corollary \ref{cr2019} not only
replaces the condition on the height of $I$
by the weaker condition, but also removes the hypothesis that
 the joint reduction is a systems of parameters for ${M}$ in \cite[Theorem 3.1]{VDT}.
    Corollary \ref{vt2019} and Corollary \ref{cr2019} seem to make the
problem of expressing mixed multiplicities into the Hilbert-Samuel
multiplicity of joint reductions become clear.

Next, to understand more the correlation between the conditions in
 Corollary \ref{vt2019}, and to consider the relationship between multiplicities of joint reductions of the same type, we would like to give the  following example.
\begin{example}\label{exam} Let $(A, \frak m)$ be a Noetherian local ring of dimension $d=4$. Assume that $x_1, x_2, x_3, x_4$ is a  system of parameters for  $A$. Set $I_1= (x_1, x_2, x_3)$, $I_2 = (x_3),$ $J = (x_1, x_2, x_3, x_4)$, ${\bf I}=I_1, I_2$ and $I=I_1I_2$. Then
 for all $n_1, n_2, n_0 \ge 1,$ we have
$$I_1^{n_1}I_2^{n_2}J^{n_0} = x_3I_1^{n_1}I_2^{n_2-1}J^{n_0} + (x_1,x_2,x_4)I_1^{n_1}I_2^{n_2}J^{n_0-1}.$$
  So ${\bf x}= x_3, x_1, x_2, x_4$ is a joint reduction  of ${\bf I}, J$ with respect to $A$ of the type $(0, 1, 3).$ In this case,
  ${\bf x_I} = x_3,$ ${\bf k} = (0, 1)$ and
     $\dim A/({\bf x_I}, I)= \dim A/({\bf x_I}).$ Hence the dimensional condition of Corollary \ref{vt2019} (ii) is not satisfied.
     Note that ${\bf x}$ is a  system of parameters for  $A.$
          Now,
    we consider the relationship between $e(J^{[3]}, \mathbf{I}^{[\mathbf{k}]}; A)$ and $e({\bf x}; A).$
                Since $0_A: x_3 \subset 0_A: I^\infty,$ it follows that $x_3$ is an $I$-filter-regular element with respect to $A.$
           Then by Corollary \ref {vt5/3} (iii) we have
   $e(J^{[3]}, \mathbf{I}^{[\mathbf{k}]}; A) = e(J^{[3]}, \mathbf{I}^{[\mathbf{0}]}; A/(x_3)).$ But since $I(A/(x_3)) = 0$, it follows that $e(J^{[3]}, \mathbf{I}^{[\mathbf{0}]}; A/(x_3))=0.$ Therefore, we get  $e(J^{[3]}, \mathbf{I}^{[\mathbf{k}]}; A) = 0 < e({\bf x}; A).$
   Next, note that
   ${\bf z} = x_3, x_1^2, x_2^2, x_4^2$ is also a joint reduction  of ${\bf I}, J$ with respect to $A$ of the type $(0, 1, 3).$ So ${\bf x}$ and ${\bf z}$ are joint reductions of the type $(0, 1, 3)$ and are systems of parameters
    for  $A.$ However,
   $e({\bf x}; A) \not= e({\bf z}; A)$. This is also a remarkable fact.
  \end{example}

Finally,  we would like to discuss some facts related to Corollary \ref {vt5/3} (iii).

Recall that an element $a \in I_i $ is called a {\it  Rees superficial element} of $\mathrm{\bf I}$ with respect to  $M$ if
 $aM \bigcap \mathbb{I}^{\mathrm{\bf n}}I_iM = a\mathbb{I}^{\mathrm{\bf n}}M$ for all
 $ \mathrm{\bf n}\gg \bf 0,$ and the following.

 \begin{definition}[\cite{Vi}(see e.g. \cite{DMT, MV, DQV, VDT})] \label{dn23} An element $x \in A$ is called a
{\it weak}-(FC)-{\it element} of $\bf I$ with respect to $M$ if
there exists $1 \le i \le d$ such that $x \in I_i$ and
the following conditions are satisfied:
 \begin{enumerate}[(FC1):]
 \item $x$ is a  Rees superficial element of $\mathrm{\bf I}$ with respect to  $M.$

\item $x$ is an $I$-filter-regular element with respect to $M,$
i.e.,\;$0_M:x \subseteq 0_M: I^{\infty}.$
 \end{enumerate}
 \end{definition}

  If ${\bf x}$ is a weak-$\mathrm{(FC)}$-sequence of $\mathrm{\bf I}, J$ with respect to $M$ of the type $(\mathrm{\bf k}, k_0),$ then
\begin{equation}\label{vtttvv}\bigtriangleup^{(k_0,\;\mathrm{\bf
k})}P(n_0, {\bf n}, J,  \mathbf{I}, M)= P(n_0, {\bf n}, J,  \mathbf{I}, M/({\bf x})M) \end{equation} by \cite[Proposition 2.5]{TV4} (see e.g. \cite [(3)]{DV} or the proof of \cite[Proposition 3.3 (i)]{MV}).

Let $\frak I_i$ be a sequence  of  elements of $I_i$
for all $1 \le i \le d$. Assume that $\frak I_1, \ldots, \frak
 I_d$ is a Rees superficial sequence of $\mathbf{I}$ with respect
to $M$. Then by \cite[Theorem
 3.4 (i)]{Vi4}(see \cite[Note 2.4 (1)]{TV4} or \cite[Proposition 2.6 (1)] {VDT}),  for all large $\bf n$, we have
\begin{equation}\label{vttt1}(\frak I_1, \ldots, \frak
 I_d)M \bigcap \mathbb{I}^{\mathbf{n}}{M}
=  \sum_{i=1}^d(\frak I_i)
\mathbb{I}^{\mathbf{n} - \mathbf{e}_i}M. \end{equation}
Now, if $\mathbb{I}^{\mathbf{n}}{M} \subseteq (\frak I_1, \ldots, \frak
 I_d)M$ for all large $\bf n,$ then $(\frak I_1, \ldots, \frak
 I_d)M \bigcap \mathbb{I}^{\mathbf{n}}{M} =\mathbb{I}^{\mathbf{n}}{M}$ for all large $\bf n.$ In this case,
$\mathbb{I}^{\mathbf{n}}{M}
=  \sum_{i=1}^d(\frak I_i)
\mathbb{I}^{\mathbf{n} - \mathbf{e}_i}M$ for all large $\bf n,$ i.e., $\frak I_1, \ldots, \frak
 I_d$ is a joint reduction  of
$\mathbf{I}$ with respect to $M.$

 A sequence ${\bf a}$ in $\mathrm{\bf I}, J$
     is called a {\it mixed  multiplicity system of $\mathrm{\bf I}, J$ with respect to $M$ of the type $(\mathrm{\bf k}, k_0)$} if
${\bf a}$ is a  Rees superficial sequence of $\mathrm{\bf I}, J$ with respect to $M$ of the type $(\mathrm{\bf k}, k_0)$ and
$\dim {M}/({\bf a})M: I^\infty\le 1$ \cite[Definition 4.4] {VT3}.

 Let $\bf a$ be a mixed  multiplicity system of $\mathrm{\bf I}, J$ with respect to $M$ of the type $(\mathrm{\bf k}, k_0).$ Then by \cite[Proposition 2.3]{VDT} (see \cite[Remark 1]{Vi}, \cite{DMT, MV, DQV}), there exists a
weak-$\mathrm{(FC)}$-element $a \in J$ of
${\bf I}, J$ with respect to $M/({\bf a})M.$
So ${\bf a}, a$ is a Rees superficial sequence of $\mathrm{\bf I}, J$ with respect to  $M,$ and
by \cite [Proposition 2.5]{TV4},
$\dim M/({\bf a}, a)M: I^\infty \le \dim M/({\bf a})M: I^\infty -1.$ Since $\bf a$ is a mixed  multiplicity system, we get $\dim {M}/{({\bf a})M: I^\infty}\le 1.$ So $\dim M/({\bf a}, a)M: I^\infty \le 0.$ Hence by Remark \ref{no4.3a} (iii),
 $I \subseteq \sqrt{\mathrm{Ann}(M/({\bf a}, a)M)}.$ Therefore, $J^{n_0}\mathbb{I}^{\mathbf{n}}M \subseteq ({\bf a}, a)M$ for all large $n_0, \bf n.$  Consequently, by (\ref{vttt1}), ${\bf a}, a$
is a joint reduction.
We have the following note.
\begin{note}\label{note.a} Any mixed  multiplicity system of $\mathrm{\bf I}, J$ with respect to $M$ of the type $(\mathrm{\bf k}, k_0)$ in \cite {htv, VT3}
is part of a joint reduction  of
$\mathbf{I}, J$ with respect to $M$
 of the type $({\bf
k},k_0+1).$
\end{note}

Assume that $e(J^{[k_0 +1]}, \mathbf{I}^{[\mathbf{k}]}; M)$ is defined, and
 let
 ${\bf x}$ be a weak-$\mathrm{(FC)}$-sequence of $\mathrm{\bf I}, J$ with respect to $M$ of the type $(\mathrm{\bf k}, k_0+1).$ Then $\bigtriangleup^{(k_0+1,\;\mathrm{\bf
k})}P(n_0, {\bf n}, J,  \mathbf{I}, M)= 0$ by (\ref{vv25}) and
 $\bigtriangleup^{(k_0+1,\;\mathrm{\bf
k})}P(n_0, {\bf n}, J,  \mathbf{I}, M)= P(n_0, {\bf n}, J,  \mathbf{I}, M/({\bf x})M)$ by (\ref{vtttvv}).
So $$P(n_0, {\bf n}, J,  \mathbf{I}, M/({\bf x})M) =0.$$
In this case, $J^{n_0}\mathbb{I}^{\mathbf{n}}M \subseteq (\mathbf{x})M$ for all large $n_0, \bf n.$
Hence by (\ref{vttt1}), ${\bf x}$
is a joint reduction  (see \cite[Remark 2.5 (iv)]{TV3}. So we obtain the following.
\begin{note} \label{note.b} If $e(J^{[k_0 +1]}, \mathbf{I}^{[\mathbf{k}]}; M)$ is defined, then any weak-$\mathrm{(FC)}$-sequence of $\mathrm{\bf I}, J$ with respect to $M$ of the type $(\mathrm{\bf k}, k_0+1)$ is a joint reduction  of
$\mathbf{I}, J$ with respect to $M$
 of the type $({\bf
k},k_0+1).$
\end{note}

By the facts just mentioned, we would like to give the following conclusions.

\begin{remark}\label{no4.2a}
In studying mixed multiplicities, one always needs  to select
an element $x \in I_i$ such that
\begin{equation}\label{07/3}
e(J^{[k_0+1]}, \mathbf{I}^{[\mathbf{k}]}; M)=
e(J^{[k_0+1]}, \mathbf{I}^{[\mathrm{\bf k} - \mathbf{e}_i]};
 M/xM).\end{equation}
  And one of the approaches to this problem is to choose $x \in I_i$ in order to generate the equation
 $P(n_0, {\bf n}, J, \mathbf{I}, M/xM)=\bigtriangleup^{(0,\; \mathbf{e}_i)}P(n_0, {\bf n}, J,
\mathbf{I},  M).$ So
  one used different sequences: Risler and Teissier in 1973 \cite{Te} used superficial sequences of $\frak m$-primary ideals; Viet in 2000 \cite{Vi} used  weak-(FC)-sequences; Trung in 2001 \cite{Tr2} used $''$bi-filter-regular sequences$''$;  Trung and Verma in 2007 \cite{TV} used $(\varepsilon_1,\ldots,\varepsilon_m)$-superficial sequences. However, \cite[Remark 3.8]{DMT} showed that the sequences used in \cite{Te, Tr2, TV} are weak-(FC)-sequences, i.e.,
sequences are both filter-regular sequences and Rees superficial sequences.
In this paper, by a different approach, to achieve (\ref{07/3}),
 we only use filter-regular sequences which are joint reductions (Corollary \ref {vt5/3} (iii)). Moreover, by Note \ref {note.a} and Note \ref {note.b}, the sequences used in \cite {Re, Te, htv, Tr2, TV, Vi,  VT3} are parts of joint reductions. Hence applying Theorem \ref{le2020}, one can easily recover early results on mixed multiplicities of ideals
in \cite {Re, Te, htv, Tr2, TV, Vi,  VT3}.
\end{remark}


\end{document}